\documentclass[10pt]{article}
\usepackage{mathrsfs}
\usepackage{amsmath}
\usepackage{amssymb}
\usepackage{subfigure}
\usepackage{amsfonts}
\usepackage{color,xcolor}
\usepackage{graphicx}
\usepackage{epstopdf}
\usepackage{geometry}
\usepackage{manfnt}
\usepackage[pagewise]{lineno} %\linenumbers
\RequirePackage[colorlinks,citecolor=blue,urlcolor=blue]{hyperref}
\newtheorem{theorem}{Theorem}[section]

\newtheorem{lemma}{Lemma}[section]
\newtheorem{proposition}{Proposition}[section]

\newtheorem{definition}{Definition}[section]
\newtheorem{remark}{Remark}[section]
\newtheorem{example}{Example}[section]
\newcommand{\beqn}{\begin{eqnarray}}
\newcommand{\eeqn}{\end{eqnarray}}
\newcommand{\beq}{\begin{equation}}
\newcommand{\eeq}{\end{equation}}

\usepackage{enumitem}
\usepackage{comment}
\begin{document}
\makeatletter %%%%%%%%%%%%%%%%%%%%%%%%%
\renewcommand{\theequation}{\thesection.\arabic{equation}}
\title{\bf A new class of pyramidal central configurations for spatial Newtonian 6-body problems}
\renewcommand{\@thesubfigure}{\hskip\subfiglabelskip}
\author{Liang Ding$^{1,2}$, Jin Wang$^{1}$, Jinrong Wang$^{2,3}$ and Zhifu Xie$^{4}$\thanks{Corresponding author.} \\
 {\small \it $^1$School of Data Science and Information Engineering, Guizhou Minzu University,} \\
 {\small \it Guiyang, Guizhou 550025, P.R. China}\\
{\small \it $^2$Guizhou Provincial Key Laboratory of Applied Mathematics and Computing Power \& Algorithms,}\\
 {\small \tt ding2016liang@126.com (Liang Ding), wangjyfy@126.com (Jin Wang)}\\
{\small \it $^3$School of Mathematics and Statistics, Guizhou University, Guiyang, 550025, P.R. China}\\
  {\small \tt jrwang@gzu.edu.cn (Jinrong Wang)}\\
 {\small \it $^4$School of Mathematics and Natural Science, }\\{\small \it The University of Southern Mississippi, Hattiesburg,
Mississippi 39406, USA}\\
 {\small \tt zhifu.xie@usm.edu (Zhifu Xie)}
 }
\date{}
 \maketitle
\noindent{\hrulefill} \vskip1mm\noindent
 {\bf Abstract}

While a non-concyclic rhombus can support a central configuration in the planar Newtonian four-body problem, such a configuration cannot serve as the base of a pyramidal central configuration in the spatial five-body problem. In this paper, we establish the existence of a new class of pyramidal central configurations in the spatial Newtonian six-body problem, characterized by a non-concyclic quadrilateral base. This geometric structure distinguishes the new configurations from the known pyramidal central configurations of the spatial five-body problem. We investigate several fundamental properties of these configurations and provide numerical examples illustrating the theoretical results in each case. Furthermore, we demonstrate that this new class of central configurations can be used to construct perverse solutions in three-dimensional space, revealing an additional connection between the geometry of central configurations and the dynamics of the Newtonian $N$-body problem.



\section{Introduction} \label{sec1}
\setcounter{equation}{0}
Central configurations play a significant role in the analysis of collision orbits, the expansion of gravitational systems, and the limitations that affect the configurations assumed by the bounded motion of the $n$-body problem \cite{Saari1980}. The motion of $n$ bodies under Newton's law of universal gravitation is governed by the following system of second-order ordinary differential equations:
\begin{equation}\label{1.1}
    m_k \ddot{q}_k=\sum\limits_{{j\neq k\atop{{1\leq j\leq n}}}}
\frac{m_{j} m_{k}(q_{j}-q_{k})}{|q_{j}-q_{k}|^{3}},\quad \  \  k=1,2,\cdots,n.
\end{equation}
Homothetic and homographic orbits represent special classes of solutions in which the mass configuration remains self-similar over time, where the configuration space is defined as:
$$
\triangle=\bigcup_{1\leq k\neq s\leq n}\triangle_{k,s}=\{x=[x_{1},\,x_{2},\,\cdots,\,x_{n}]\in\mathbb{R}^{dn}\,|\,x_{k}\neq x_{s},\,1\leq k\neq s\leq n\},
$$
where $d = 2$ corresponds to a planar configuration and $d = 3$ to a spatial configuration.
%The study of these types of solutions has long attracted considerable interest in celestial mechanics.
Both of these two types of solutions arise from central configurations--specific arrangements of masses in which the gravitational acceleration on each body is proportional to its position vector relative to the center of mass.
In this paper, we investigate a class of spatial central configurations.  For $n$ positive masses $m_{1},\,m_{2},\,\cdots,\,m_{n}$ with corresponding positions $q=[q_{1},\,q_{2},\,\cdots,\,q_{n}]$, we first introduce the definition of central configuration \cite{Wintner1941} as follows.
\begin{definition}\label{def1.1}
A configuration $q=[q_{1},\,q_{2},\,\cdots,\,q_{n}]\in \triangle$ is called a \textbf{central configuration (CC)} for a positive mass vector $m=(m_1, m_2, \dots, m_n)$ if
there exists a constant $\lambda\in\mathbb{R}$ such that
\begin{eqnarray}\label{1.2}
\sum\limits_{{j\neq k\atop{{1\leq j\leq n}}}}
\frac{m_{j} m_{k}(q_{j}-q_{k})}{|q_{j}-q_{k}|^{3}}+\lambda m_{k}(q_{k}-q_{0})=0, \quad \  \  k=1,2,\cdots,n,
\end{eqnarray}
where the barycenter $q_{0}$ is given by $q_{0}=(\sum_{1\leq j\leq n}m_{j}q_{j})/M$ with $M=\sum_{1\leq j\leq n}m_{j}$.
\end{definition}

The study of central configurations in the Newtonian $n$-body problem dates back to the classical work of Euler. In 1767, Euler \cite{Euler1767} discovered the well-known collinear central configurations for the 3-body problem. In 1772, Lagrange \cite{Lagrange1772} identified the equilateral-triangle central configuration for the planar 3-body problem. % n both the Euler collinear and the Lagrange equilateral central configurations, the three masses forming the configuration may not be equal to each other.
 In 1985, Perko and Walter \cite{Perko1985} proved that for any $n \geq 4$, all masses in a regular $n$-gon central configuration must be equal. Wang later extended this result in 2019 \cite{Wang2019}, generalizing it from Newtonian potentials to arbitrary positive homogeneous potentials. For further reading, see \cite{Fernandes2017,Perez2007,Santoprete2021a,Santoprete2021b,Wang2023}.

\begin{figure}[htbp]\centering
\begin{minipage}[t]{0.42\linewidth}
\includegraphics[width=6.5cm,height=7.5cm]{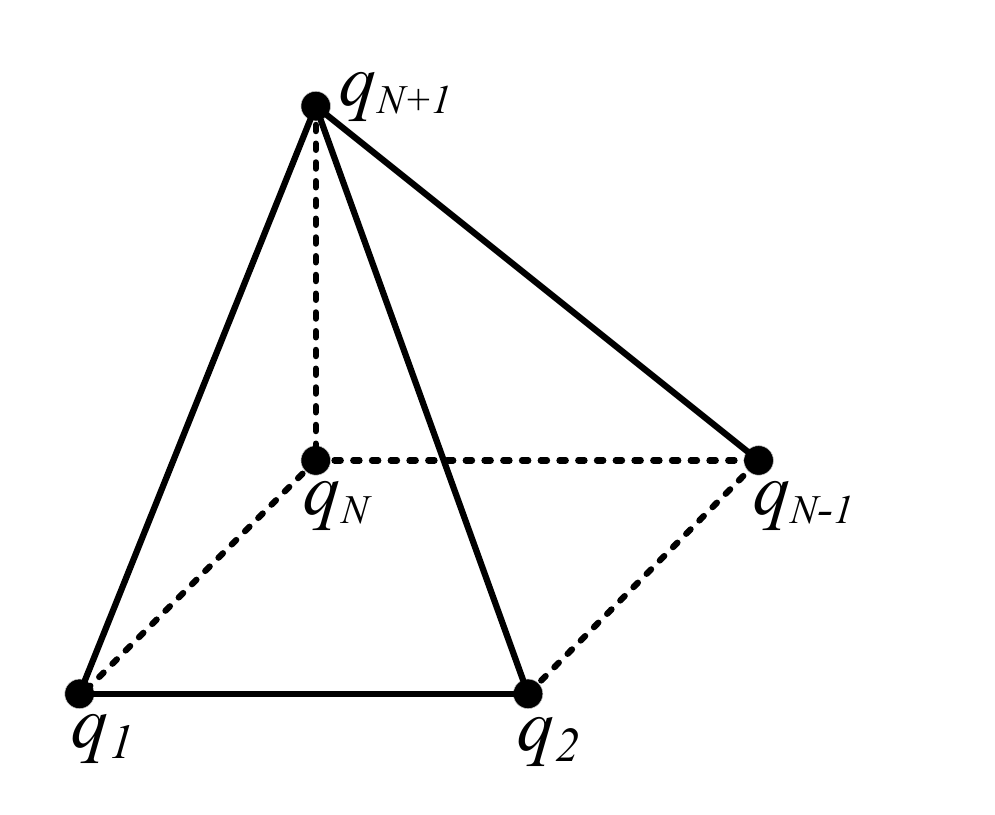}
\caption{Spatial ($N$+1)-body problems: Single Pyramidal CC}
\label{fig1}
\end{minipage}%
\ \ \ \
\begin{minipage}[t]{0.42\linewidth}
\includegraphics[width=8cm,height=8cm]{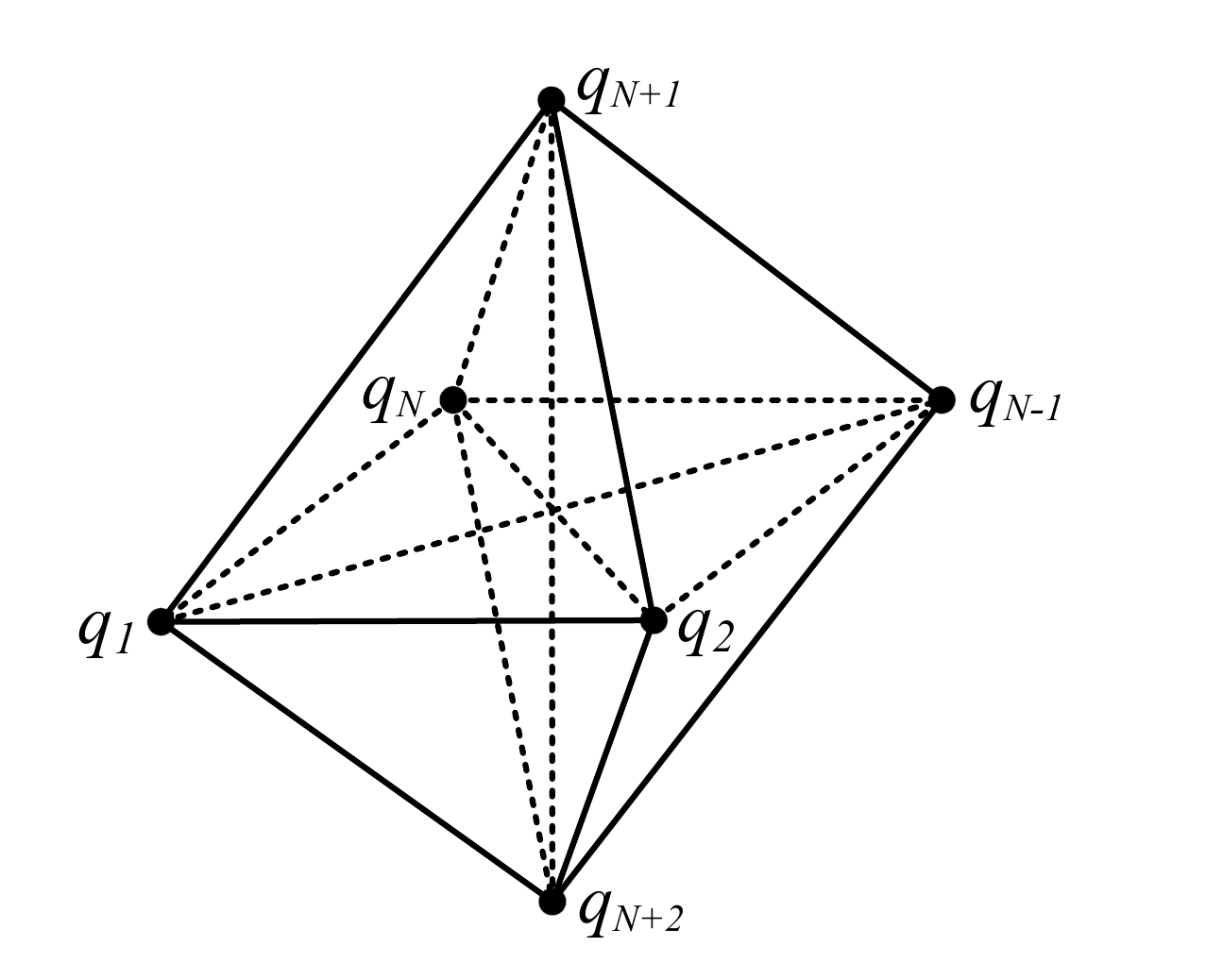}
\caption{Spatial ($N$+2)-body problems: Bi-pyramidal CC}
\label{fig2}
\end{minipage}
\end{figure}

For spatial central configurations, the simplest type is the single pyramidal central configuration, and its definition is described as the following \cite{Faycal1996}.
\begin{definition}\label{def1.2}
 %A central configuration of $N$+1 masses ($n=N+1$), $N$ of them are coplanar and the ($N$+1)-th mass is off the plane, is called a single pyramidal central configuration. Equivalently, we will say that the central configuration has the shape of a pyramid, with the $N$ coplanar masses forming the basis of the single pyramidal central configuration (see Figure {\color{blue}1}).
A central configuration of \( N+1 \) masses (\( n = N+1 \)), where \( N \) of them are coplanar and the (\( N+1 \))-th mass is positioned off the plane, is called a \textbf{single pyramidal central configuration}. Equivalently, we describe such a central configuration as having the shape of a pyramid, with the \( N \) coplanar masses forming the base of the pyramid (see Figure \ref{fig1}).
\end{definition}

When $n=4$, Lehmann-Filh\'{e}s and Pizzetti \cite{Lehmann1891,Pizzetti1904} demonstrated that the regular tetrahedron is the unique single pyramidal central configuration, having an equilateral triangle as its base. For $n=5$, Fay\c{c}al \cite{Faycal1996} proved that there exist exactly three types of single pyramidal central configurations, and all of the bases are concyclic.

Also for the case of $n=5$, Mello and Fernandes \cite{Mello2011} obtained the existence of a class of bi-pyramidal central configurations with an equilateral triangular base. The definition of a bi-pyramidal central configuration is as follows:
\begin{definition}\label{def1.3}
 %A central configuration of $N$+2 masses ($n=N+2$), where $N$ of them are coplanar and the remaining two masses are off the plane, is called a bi-pyramidal central configuration  (see Figure {\color{blue}2}).

 A central configuration of \( N+2 \) masses (\( n = N+2 \)), where \( N \) of them are coplanar and the remaining two masses are positioned off the plane, is called a \textbf{bi-pyramidal central configuration} (see Figure \ref{fig2}).

\end{definition}

Note that in the well-known work of \cite{Longley1907}, Longley discovered the existence of planar central configurations in the Newtonian four-body problem, where four masses are placed at the vertices of a non-concyclic rhombus. Moreover, in the spatial Newtonian five-body problem, if the base is not concyclic, the five masses cannot form a single pyramidal central configuration \cite{Faycal1996}.  Additionally, in spatial Newtonian \( n \)-body problems, all previously reported bi-pyramidal central configurations \cite{Corbera2013, Leandro2003, Mello2011, Zhang2001} have a concyclic base.

In this paper, we investigate a new class of central configurations in Newtonian 6-body problems. More precisely, we consider configurations where masses \( m_1, m_2, m_3, m_4 \) are positioned counterclockwise at the vertices of a rhombus  with corresponding positions \( q_1, q_2, q_3, q_4 \), while at least one additional mass \( m_k \) (with position \( q_k \), where \( k\in\{5,6\} \)) is located outside the plane containing the rhombus.

Moreover, recall that central configurations can generate homographic/homothetic solutions to the Newtonian $n$-body problem, let $\bar{q}(t) = (\bar{q}_{1}(t), \bar{q}_{2}(t), \ldots, \bar{q}_{n}(t))$ be a solution corresponding to masses  $(m_{1}, m_{2}, \ldots,$ $ m_{n})$ under Newtonian equations (\ref{1.1}). In 2001, Chenciner~\cite{Chenciner2004} posed two questions: (Q1)  Does there exist another set of masses \((m^{\prime}_{1}, m^{\prime}_{2}, \ldots, m^{\prime}_{n})\) such that \(\bar{q}(t)\) remains a solution? (Q2) The same as (Q1) but insisting that the total mass and the center of mass remain unchanged. If the answer to the first question (resp. the second) is yes, \(\bar{q}(t)\) is called a \emph{perverse solution} (resp. a \emph{really perverse solution}). He demonstrated that for planar Newtonian \(n\)-body problems with \(n \geq 3\), such perverse solutions do exist. In 2004, Ouyang, Xie, and Zhang~\cite{Ouyang2004} constructed perverse solutions in the spatial case for \(n = 5, 6, 7, 8, 9, 10\) using single pyramidal central configurations. More recently, in 2025, Bang, Chenciner, and Sim\'o~\cite{Bang2025} proved the existence of really perverse solutions in the planar case, based on nested regular polygon configurations with one mass at the center. In section \ref{sec9}, we prove the existence of perverse solutions in the spatial case, constructed from the new class of central configurations in the 6-body problems. This configuration is unique upto rotation, scaling and translation, but these perverse solutions are not really perverse.

\section{Main results}\label{sec2}\setcounter{equation}{0}

\subsection{Configurations with \texorpdfstring{$q_5$}{R} and  \texorpdfstring{$q_6$}{R}  on the opposite sides of the rhombus plane }

\begin{figure}[htbp]\centering
\begin{minipage}[t]{0.45\linewidth}
\includegraphics[width=8cm,height=8.5cm]{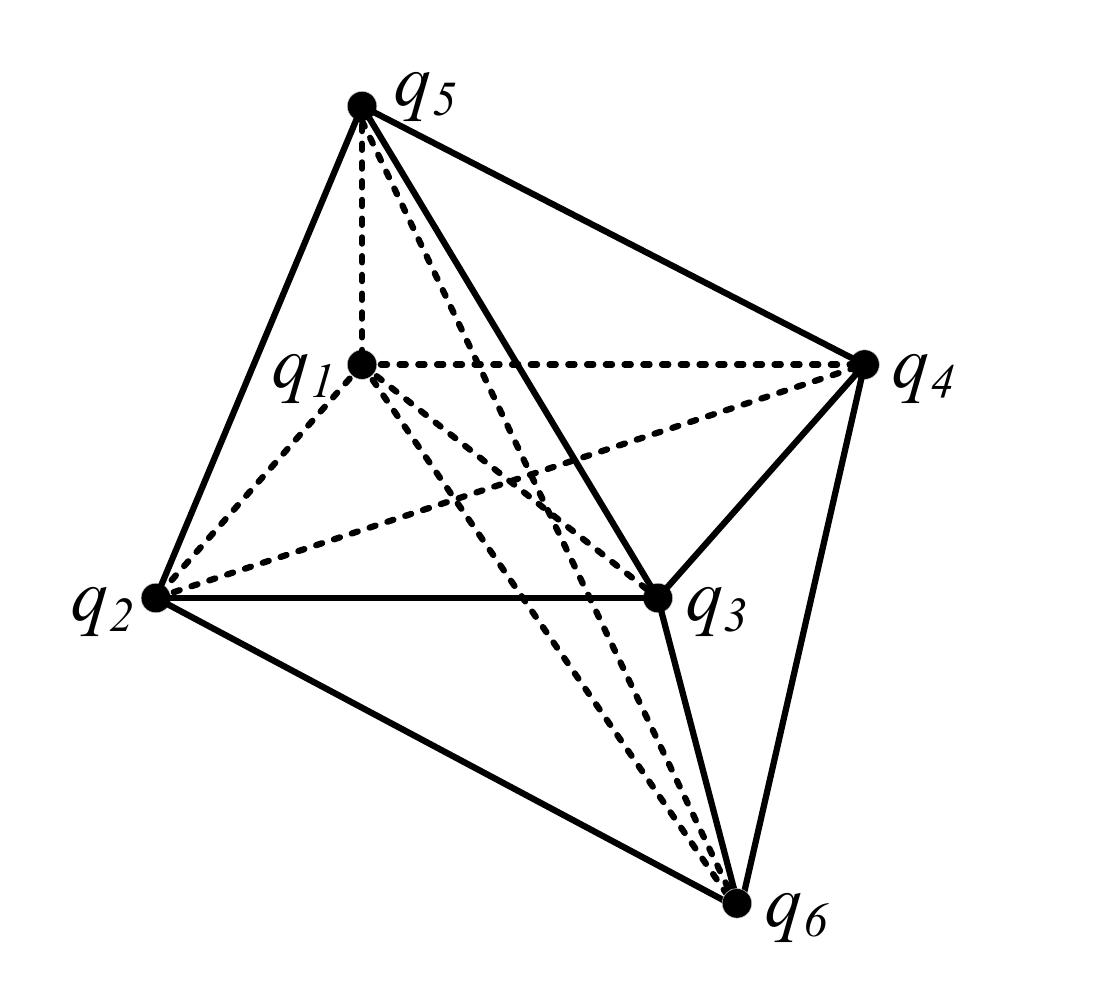}
\caption{Spatial 6-body problems: $q_{5}$ and $q_{6}$ lie on opposite sides of the rhombus}
\label{fig3}
\end{minipage}%
\ \ \ \
\begin{minipage}[t]{0.45\linewidth}
\includegraphics[width=8cm,height=8.5cm]{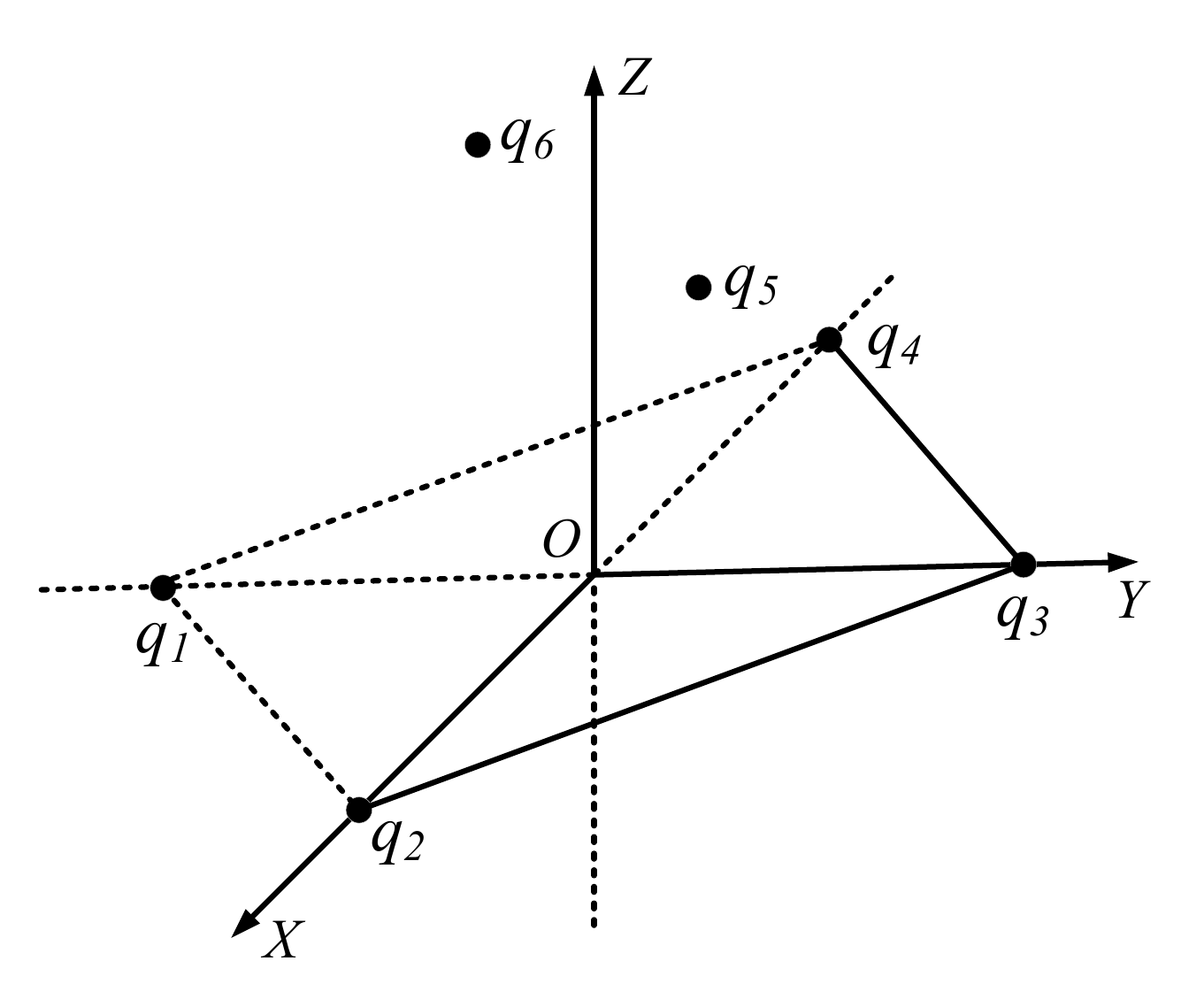}
\caption{Spatial 6-body problems: $q_{5}$ and $q_{6}$ lie on the same side of the rhombus but outside its base plane}
\label{fig4}
\end{minipage}
\end{figure}

\begin{figure}[htbp]\centering
\begin{minipage}[t]{0.45\linewidth}
\includegraphics[width=8cm,height=8.5cm]{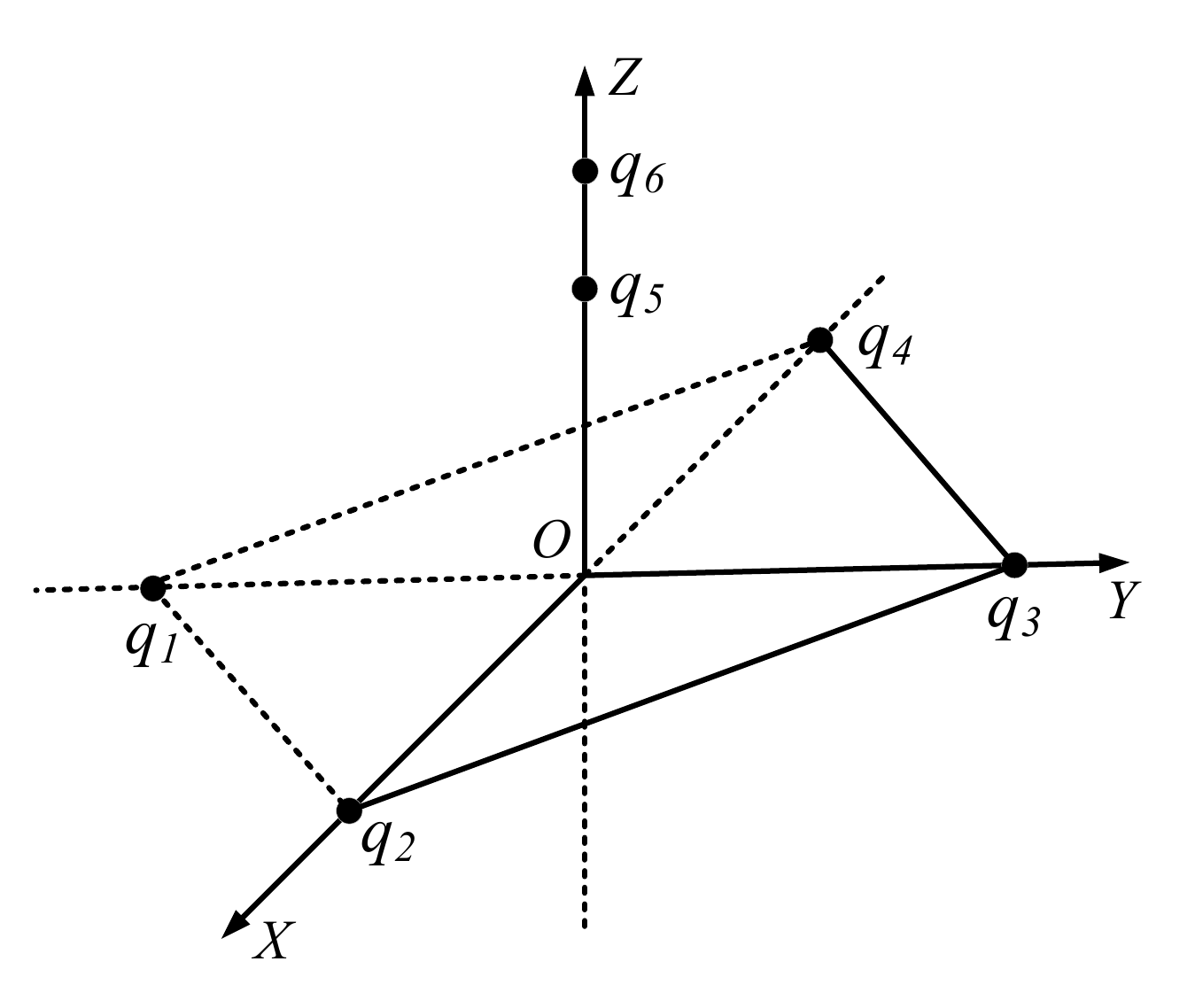}
\caption{Spatial 6-body problems: $q_{5}$ and $q_{6}$ lie on the same side of the rhombus but outside its base plane, with $q_{5}$ lies on the $Z$-axis}
\label{fig5}
\end{minipage}%
\ \ \ \ \ \ \
\begin{minipage}[t]{0.45\linewidth}
\includegraphics[width=8.2cm,height=8.6cm]{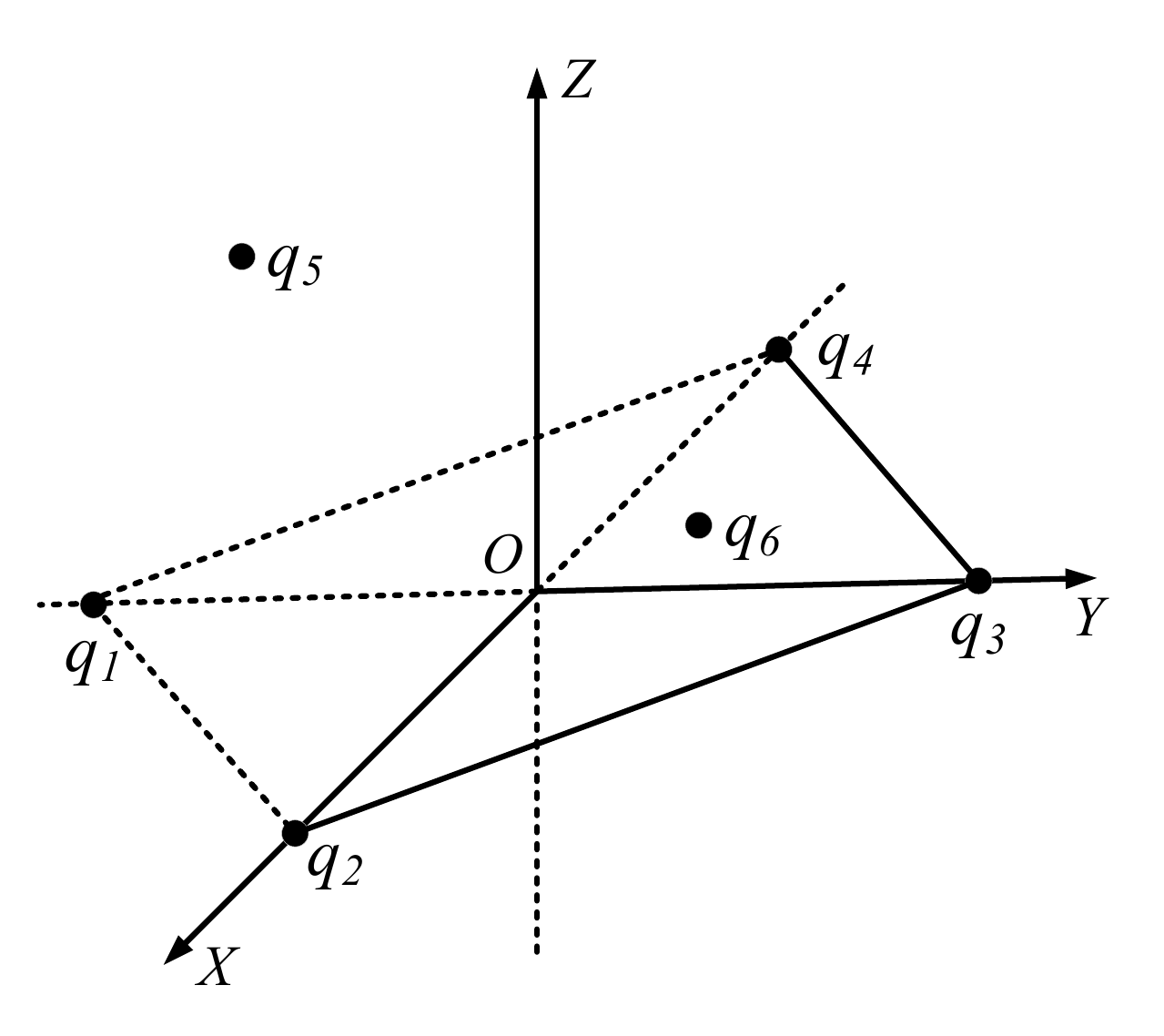}
\caption{Spatial 6-body problems: $q_{6}$ lies in the plane containing the rhombus (This CC does not exist)}
\label{fig6}
\end{minipage}
\end{figure}

We consider the two masses located off the plane containing the rhombus (see Figure \ref{fig3}). The diagonals \( q_{1}q_{3} \) and  \( q_{2}q_{4} \) intersect at the origin
$O=(0,0,0)$, and the positions $[q_{1}, q_{2},\cdots,q_{5},q_{6}]$ of the six positive masses $(m_{1}, m_{2},\cdots,m_{5},m_{6})$ are given by:
\begin{eqnarray}\label{2.1}
\left\{
  \begin{array}{ll}
q_1=(0,-a, 0), \,\, q_2=(1,0,0), \,\,q_3=(0, a, 0), \\[0.2cm]
q_4=(-1,0,0), \,\,q_5=\left(x_5, y_5, z_5\right),\,\, q_6=\left(x_6, y_6, z_6\right),
  \end{array}
\right.
\end{eqnarray}
where $a>0$, $x_5, y_5, x_6,y_6\in\mathbb{R}$, $z_5>0$ and $z_6<0$. Then, we have the following two results.

\begin{theorem}\label{thm2.1}
Let the configuration \( q = [q_1, q_2, \cdots, q_6] \) be defined as in equation (\ref{2.1}) with $a>0$, $x_5, y_5, x_6,y_6\in\mathbb{R}$, $z_5>0$ and $z_6<0$. If \( q \) is a pyramidal central configuration for the mass vector \( m = (m_1, m_2, \cdots, m_6) \), then only the following three cases hold:
 \begin{enumerate}
    \item \( x_5 x_6>0, y_5 = y_6 = 0.\)  In this case \(m_3=m_1. \)
    \item  \( y_5 y_6>0, x_5 = x_6 = 0.\) Then \( m_4=m_2. \)
    \item   \(  x_5 = x_6 = y_5=y_6=0.\) Then \( m_4=m_2,\) and  \(m_3=m_1. \)
\end{enumerate}

\end{theorem}

\begin{remark}\label{rem2.1}
It is well known that for the single pyramidal central configurations in Newtonian 5-body problems, no such configuration exists when four masses are positioned at the vertices of a non-square rhombus.
However, for bi-pyramidal central configurations of Newtonian 6-body problems, such configurations can occur even when the four base masses form a non-square rhombus, and we present two examples below.

\begin{example}\label{exa2.1} \textbf{($a\neq1$).} In equations (\ref{1.2}), set $n=6$, $x_{5}=x_{6}=0.1$, $y_{5}=y_{6}=0$, $z_{5}=-z_{6}=0.86628148$, $a=0.9$, $m_{1}=m_{3}=1$, $m_{2}=0.258539199$, $m_{4}=0.853441220$, $m_{5}=m_{6}=1.045160300$. Then, according to Definition \ref{def1.1}, equation (\ref{2.1}) and direct computation using Maple, there exists $\lambda=1.861715304$ such that the six masses form a bi-pyramidal central configuration.
\end{example}

\begin{example}\label{exa2.2} \textbf{($a=1$).} In equations (\ref{1.2}), we take $n=6$, $x_{5}=x_{6}=0.1$, $y_{5}=y_{6}=0$, $z_{5}=-z_{6}=0.867472$, $a=1$, $m_{1}=m_{3}=1$, $m_{2}=0.541763934$, $m_{4}=1.49858863$, $m_{5}=m_{6}=1.4961658$. Then, by Definition \ref{def1.1}, equation (\ref{2.1}) and direct computation with Maple, there exists $\lambda=2.250204683$ such that the six masses form a bi-pyramidal central configuration.
\end{example}
\end{remark}

 If \(  x_5 = x_6 = y_5=y_6=0\), we have the following theorem.
\begin{theorem}\label{thm2.2}
Let the configuration \( q = [q_1, q_2, \cdots, q_6] \) be defined as in equation (\ref{2.1}) with \( x_5 = y_5 = x_6 = y_6 = 0,  \, z_5>0 \) and $z_6<0$, meaning that \( m_5 \) and \( m_6 \) lie on the symmetric \( z \)-axis. Assume that \( q \) is a bi-pyramidal central configuration for the mass vector \( m = (m_1, m_2, \cdots, m_6) \). Then, the following statements hold:
\begin{enumerate}
    \item If \( z_6 = -z_5 \), then \( m_3 = m_1 \), \( m_4 = m_2 \), and \( m_6 = m_5 \).
    \item If \( z_6 \neq -z_5 \), then \( a = 1 \) (the rhombus must be a square) and \( m_1 = m_2 = m_3 = m_4 \).
\end{enumerate}
\end{theorem}

\begin{remark}\label{rem2.2}
Regarding Theorem \ref{thm2.2}, we can illustrate with two examples such that the conditions
$z_{5} =-z_{6}$ and $z_{5}\neq -z_{6}$ are satisfied, respectively.

\begin{example}\label{exa2.3} \textbf{($z_{5} =-z_{6}$).} In equations (\ref{1.2}), we set $n=6$, $x_{5}=x_{6}=y_{5}=y_{6}=0$, $z_{5}=-z_6=0.78554675$,  $a=0.9291884$, $m_{1}=m_{3}=2, $ $m_{2}=m_{4}=1$, $m_{5}=m_{6}=3$. Then, there exists $\lambda=4.740327950$ such that the six masses form a pyramidal central configuration.
\end{example}

\begin{example}\label{exa2.4} \textbf{($z_{5}\neq -z_{6}$).} In equations (\ref{1.2}), we set $n=6$, $x_{5}=x_{6}=y_{5}=y_{6}=0$, $z_{5}=1.2$, $z_{6}=-1.1$, $a=1$, $m_{1}=m_{2}=m_{3}=m_{4}=1$, $m_{5}=0.691189971$, $m_{6}=0.216916499$. Then, there exists $\lambda=1.204479102$ such that the six masses form a pyramidal central configuration.
\end{example}
\end{remark}

Moreover, for the existence of the bipyramidal central configuration with a non-concyclic quadrilateral base in Theorem \ref{thm2.2}, we prove that

\begin{theorem}\label{thm2.3}
Let the configuration \( q = [q_1, q_2, \cdots, q_6] \) be defined as in equation (\ref{2.1}) with $a>0$, $x_5, y_5, x_6,y_6\in\mathbb{R}$, $z_5>0$ and $z_6<0$. For any three pairwise distinct positive numbers $m_{1}=m_{3}>0$, $m_{2}=m_{4}>0$ and $m_{5}=m_{6}>0$, there exist \( x_5 = y_5 = x_6 = y_6 = 0 \),  $z_6 = -z_5$ where $z_{5}\in\{z_{5}\in\mathbb{R}^{+}: z_{5}\neq1 \  {\rm and} \ z_{5} \neq a\}$, $0<a\neq1$ and $\lambda\in\mathbb{R}$ such that the six masses form a central configuration.
\end{theorem}

\subsection{Configurations with \texorpdfstring{$q_5$}{R} and  \texorpdfstring{$q_6$}{R}  on the same sides of the rhombus plane}

We consider configurations where at least one mass \( m_k \) ( with position \( q_k \), for \( k\in\{5,6\} \)) lies outside the plane containing the rhombus (see Figures \ref{fig4}--\ref{fig6}). Without loss of generality, we assume that $z_{5}>0$ and $z_{6}\geq0$.

\begin{theorem}\label{thm2.4}
Let the configuration \( q = [q_1, q_2, \cdots, q_6] \) be defined as in equation (\ref{2.1}) with $a>0$, $x_5,\,x_6,\,y_5,\,y_6\in\mathbb{R}$, $z_{5}>0$ and $z_{6}\geq0$. If the rhombus is square (i.e. $a=1$) and the six masses form a central configuration, then \(  x_5 = x_6 = y_5=y_6=0\) and \(  z_{6}>0\).
\end{theorem}

\begin{remark}\label{rem2.3}
As shown in Example \ref{exa2.2},  when the rhombus is a square, and \( q_5 \) and \( q_6 \) are positioned on opposite sides of this plane containing the square, the line segment $q_{5}q_{6}$ does not pass through the square's center.
In contrast, Theorem \ref{thm2.4} states that if \( q_5 \) and \( q_6 \) are positioned on the same side of the plane, then $q_{5}q_{6}$ must pass through the center. Here is an example.

\begin{example}\label{exa2.5} In equation (\ref{1.2}), set $n=6$, $x_{5}=x_{6}=y_{5}=y_{6}=0$, $z_{5}=1$, $z_{6}=2$, $a=1$, $m_{1}=m_{2}=m_{3}=m_{4}=1$, $m_{5}=1.854247192$, $m_{6}=0.5020077162$. Then, there exists $\lambda=1.657583098$ such that the six masses form a pyramidal central configuration.
\end{example}

\end{remark}

Moreover, we have the following theorems.

\begin{theorem}\label{thm2.5}
Let the configuration \( q = [q_1, q_2, \cdots, q_6] \) be defined as in equation (\ref{2.1}) with $a>0$, $x_5,\,x_6,\,y_5,\,y_6\in\mathbb{R}$, $z_{5}>0$ and $z_{6}\geq0$. When the six masses form a central configuration where the rhombus is a non-square, we have the following results.
\begin{enumerate}
    \item The points \( q_5 \) and \( q_6 \) both lie either in the $XOZ$-plane or in the $YOZ$-plane.
    
   \item If $q_5$ and $q_6$ belong to the $XOZ$-plane, then when $q_{5}$ does not lie on the positive $Z$-axis, then the rhombus must be a non-square, and $\sqrt{3}/3<a<1$; when $q_{5}$ lies on the positive $Z$-axis, then the rhombus must be a square ($a=1$).
   In addition, when the rhombus is a non-square, the coordinates satisfy $x_5=-x_6\neq0$, $y_5=y_6=0$ and $z_{5}=z_{6}>0$, and the masses satisfy $m_{1}=m_{3}$, $m_{2}=m_{4}$ and $m_{5}=m_{6}$. 
   
   \item If $q_5$ and $q_6$ belong to the $YOZ$-plane, then when $q_{5}$ does not lie on the positive $Z$-axis, then the rhombus must be a non-square, and $1<a<\sqrt{3}$; when the rhombus is a non-square, then the rhombus must be a square ($a=1$).
   In addition, when the rhombus is a non-square, the coordinates satisfy $x_5=x_6=0$, $y_5=-y_6\neq0$ and $z_{5}=z_{6}>0$, and the masses satisfy $m_{1}=m_{3}$, $m_{2}=m_{4}$ and $m_{5}=m_{6}$.
\end{enumerate}
\end{theorem}

For central configurations with a non-square base, in addition to the necessary conditions for existence given in Theorem \ref{thm2.5}, we have the following sufficient conditions for existence.
central configuration:
\begin{theorem}\label{thm2.6}
 For every $a\in(\sqrt{3}/3,\,1) \cup (1, \,\sqrt{3})$, there exists a central configuration \( q = [q_1, q_2, \cdots, q_6] \) satisfying equation (\ref{1.2}). 
\end{theorem}
\begin{example}\label{exa2.6} In equation (\ref{1.2}), set $n=6$, $x_{5}=-x_{6}=0.37571841$,  $y_{5}=y_{6}=0$, $z_{5}=z_{6}=1.160022542$, $a=0.85$, $m_{1}=m_{3}=1$, $m_{2}=m_{4}=0.7326319459$, $m_{5}=m_{6}=0.1262842998$.
Then there exists $\lambda=1.132147153$ such that the six masses form a pyramidal central configuration.
\end{example}

\subsection{Existence of a perverse solution in three-dimensional space}

We construct a perverse solution generated from a central configuration in Theorem \ref{thm2.4}. More precisely, suppose that the position vector of $m_1,\,m_2,\,\ldots,\,m_6$ are as follows:
\begin{eqnarray}\label{2.2}
\left\{
  \begin{array}{ll}
q_1= (0, -1,0),\,\,q_2= (1, 0,0),\,\, q_3= (0, 1,0), \\ [0.3cm]
q_4= (-1, 0, 0),\,\,q_5= (0,0,z_{5}),\,\,q_6=(0,0,z_{6}),
  \end{array}
\right.
\end{eqnarray}
where $z_{6}>z_{5}>0$. Suppose that $m_1=m_2=m_3=m_4=m_0>0$. Based on the results in \cite{Ouyang2004}, in order to generate a perverse solution, the configuration $q$ must be a central configuration for at least two sets of masses and it has the following properties: \begin{itemize}
    \item The center of mass is at $q_5$.
    \item The distance $\|q_i-q_5\|$ between $q_i$ and $q_5$ is the same for all $i=1,2,3,4,6$.
    \item The parameter $\lambda$ should be same for the different set of masses.
\end{itemize}

\begin{theorem}\label{thm2.7}
There exists a $b_0 \in \mathbb{R}^+$, which is the unique positive solution for the equation
\begin{equation}\label{2.3}
    -4+\left(\sqrt{b^{2}+1}\, b +b^{2}+1\right) \left(\sqrt{2}+\frac{1}{2}\right) \sqrt{2 b^{2}+2+2 \sqrt{b^{2}+1}\, b}=0.
\end{equation}
Let $z_5=b_0$ and $z_6=\sqrt{{b_0}^{2}+1}+b_0$. Then the configuration \( q = [q_1, q_2, \cdots, q_6] \)  defined as in equation (\ref{2.2}) is a central configuration for the mass vector $(m_1, m_2, m_3, m_4, m_5, m_6)$ with $m_1=m_2=m_3=m_4=m_0$, $m_6=\frac{b_0m_0}{\sqrt{{b_0}^2+1}}$, where $m_0$ and $m_5$ are arbitrary positive values. The corresponding parameter $\lambda$ is given by
\begin{equation}\label{2.4}
    \lambda= \frac{m_0\sqrt{2}}{2}+\frac{m_0}{4}+\frac{\mathit{m_5}}{\left({b_0}^{2}+1\right)^{\frac{3}{2}}}+\frac{4 bm_0}{\sqrt{{b_0}^{2}+1}\, \left(\left(\sqrt{{b_0}^{2}+1}+b_0 \right)^{2}+1\right)^{\frac{3}{2}}}.
\end{equation}
The homothetic collision/expansion solution generated from this central configuration for a fixed $\lambda$ is a perverse solution, but it is not a really perverse solution.
\end{theorem}
\begin{remark}
    Numerically the equation (\ref{2.3}) has a unique positive solution $b_0\approx 0.2354261907$. Then $z_5=0.2354261907$, $z_6= 1.262765225,$ $m_6= 0.9166445852m_0$ and $\lambda = 1.176438468m_0 + 0.9222711668m_5$. Then this $q$ is a central configuration for the mass vector $[m_0, m_0, m_0, m_0, m_5,$ $0.9166445852m_0],$ which has two free variables $m_0$ and $m_5.$ For this configuration, center of mass is at $q_5$, which is independent of the choice of $m_0$ and $m_5$. The distances $D_{15}$, $D_{25},$  $D_{35},$  $D_{45}$ and $D_{56}$  are all equal to $\sqrt{{b_0}^2+1}= 1.027339034$. For example, for fixed $\lambda =5.119690436$, $q$ is a central configuration for both the mass vector $[2, 2, 2, 2, 3, 1.833289170]$ and the mass vector $[1,1,1,1,4.275588471,0.9166445852].$ In general, for fixed $\lambda =5.119690436$, $m_0$ can be any positive value in $(0,4.351855686)$ and
    $m_5= \frac{\lambda-1.176438468m_0}{ 0.9222711668}>0$. It is clear that the sum of masses for this two set of masses are different. Therefore, the homothetic collision/expansion solution generated from this central configuration for a fixed $\lambda$ is a perverse solution, but it is not a really perverse solution.
\end{remark}

\section{Proof of Theorem \ref{thm2.1}}\label{sec3}\setcounter{equation}{0}

Let $f_{x_k}=0$ be the equation of central configuration (\ref{1.2}) for $x_k$ where $k=1,2, \cdots, 6$. That is,
\begin{eqnarray*}
f_{x_k}= \sum\limits_{{j\neq k\atop{{1\leq j\leq 6}}}}\frac{m_{j} (x_{j}-x_{k})}{D_{kj}^{3}}+\lambda (x_{k}-x_{0}),
\end{eqnarray*}
where $D_{kj} = |q_k-q_j| $ is the distance between $q_k$ and $q_j,$  and $x_0= (\sum_{1\leq j\leq 6}m_j x_j)/(\sum_{1\leq j\leq 6} m_j)$ is the center of mass in $x$-component. Similarly, the equations $f_{y_k}=0 $ and $f_{z_k}=0$ define the central configurations for the $y_k$ and $z_k$ components, respectively. There are a total of 18 equations. Then, Theorem \ref{thm2.1} can be rewritten as the following lemma.

\begin{lemma} \label{lem3.1}
Let the configuration \( q = [q_1, q_2, \cdots, q_6] \) be defined as in equation (\ref{2.1}). If \( q \) is a bi-pyramidal central configuration for the mass vector \( m = (m_1, m_2, \dots, m_6) \), then the following two conditions must be satisfied:
\begin{enumerate}[label=(\Alph*)]
    \item $y_5 y_6>0$ or $y_5=y_6=0$.
    \item $x_5 x_6>0$ or $x_5=x_6=0$.
  \end{enumerate}
  These conditions lead to four cases but only three are possible:
  \begin{enumerate}[label=(\roman*)]
    \item $y_5 y_6>0$ and $x_5 x_6>0$ (impossible).
    \item $x_5 x_6>0$ and $y_5=y_6=0$ (possible). Then $m_3=m_1.$
     \item $y_5 y_6>0$ and $x_5=x_6=0$ (possible). Then $m_4=m_2.$
    \item$x_5=x_6=0$ and $y_5=y_6=0$ (possible). Then $m_3=m_1$ and $m_4=m_2$.
  \end{enumerate}
\end{lemma}

%{\bf Proof.}\\
{\bf Proof of (A):}
Considering  $f_{z_3}-f_{z_1}=0$,  we have
\begin{equation}\label{3.1}
m_5 z_5\left(\frac{1}{D_{53}^3}-\frac{1}{D_{51}^3}\right)+m_6 z_6\left(\frac{1}{D_{63}^3}-\frac{1}{D_{61}^3}\right)=0.
\end{equation}
Since $m_5 z_5>0$ and $m_6 z_6<0$,  both $\left(1/D_{53}^3-1/D_{51}^3\right)$ and $\left(1/D_{63}^3-1/D_{61}^3\right)$ are zero or have the same sign. \\
Using equation (\ref{2.1}), then
$$D_{51}=\left[{x_5}^{2}+(y_5+a)^2+{z_5}^2\right]^{1/2}, \hspace{0.5cm}D_{61}=\left[{x_6}^{2}+(y_6+a)^2+{z_6}^2\right]^{1/2}, $$
$$D_{53}=\left[{x_5}^{2}+(y_5-a)^2+{z_5}^{2}\right]^{1/2}, \hspace{0.5cm}D_{63}=\left[{x_6}^{2}+(y_6-a)^2+{z_6}^{2}\right]^{1/2}.$$
Noticed that $a>0$. It is easy to obtain $y_5=y_6=0$ if both $\left(1/D_{53}^3-1/D_{51}^3\right)=0$ and $\left(1/D_{63}^3-1/D_{61}^3\right)=0$. Furthermore,
$y_5>0$ and $y_6>0$ if both $\left(1/D_{53}^3-1/D_{51}^3\right)>0$ and $\left(1/D_{63}^3-1/D_{61}^3\right)>0$, while
$y_5<0$ and $y_6<0$ if both$\left(1/D_{53}^3-1/D_{51}^3\right)<0$ and $\left(1/D_{63}^3-1/D_{61}^3\right)<0$.
So, $y_5=y_6=0$ or $y_5y_6>0$.\\

{\bf Proof of (B):} Considering  $f_{z_4}-f_{z_2}=0$, we have
\begin{equation}\label{3.2}
m_5 z_5\left(\frac{1}{D_{54}^3}-\frac{1}{D_{52}^3}\right)+m_6 z_6\left(\frac{1}{D_{64}^3}-\frac{1}{D_{62}^3}\right)=0.
\end{equation}
Using equation (\ref{2.1}), the distances are given by:
$$D_{52}=\left[(x_5-1)^2+{y_5}^2+{z_5}^{2}\right]^{1/2}, \hspace{0.5cm}D_{62}=\left[(x_6-1)^2+{y_6}^{2}+{z_6}^{2}\right]^{1/2}, $$
$$D_{54}=\left[(x_5+1)^2+{y_5}^{2}+{z_5}^{2}\right]^{1/2}, \hspace{0.5cm}D_{64}=\left[(x_6+1)^2+{y_6}^{2}+{z_6}^{2}\right]^{1/2}. $$
Similar to those in the {\bf Proof of (A)}, the arguments can be used to show that either $x_5 x_6>0$ or $x_5=x_6=0$.

Since equations (\ref{3.1}) and (\ref{3.2}) must hold for a pyramidal central configuration, the necessary conditions can be categorized into the four cases presented in the lemma. We now proceed to demonstrate the impossibility of these cases. \\
%\begin{lemma}\label{case1}
%In the case of  $y_5 y_6>0$ and $x_5 x_6>0$ , the six masses cannot form a double pyramidal central configuration as defined in %(\ref{1.2}).
%\end{lemma}

{\bf Proof of the impossible case (i):}
If $x_5x_6>0$,  then  $\left(1/D_{54}^3-1/D_{52}^3\right)$ and $\left(1/D_{64}^3-1/D_{62}^3\right)$ must be either both positive or both negative. From
$$f_{y_4}-f_{y_2}= m_5 y_5\left(\frac{1}{D_{54}^3}-\frac{1}{D_{52}^3}\right)+m_6 y_6\left(\frac{1}{D_{64}^3}-\frac{1}{D_{62}^3}\right)=0,$$
we deduce that $y_5y_6<0$. This is a contradiction.

If  $y_5 y_6 > 0 $, we observe that both terms
$$
\left(\frac{1}{D_{53}^3} - \frac{1}{D_{51}^3} \right) \quad \hbox{and} \quad \left(\frac{1}{D_{63}^3} - \frac{1}{D_{61}^3} \right)
$$
are either both positive or both negative. From the equation
$$ f_{x_3} - f_{x_1} = m_5 x_5 \left( \frac{1}{D_{53}^3} - \frac{1}{D_{51}^3} \right) + m_6 x_6 \left( \frac{1}{D_{63}^3} - \frac{1}{D_{61}^3} \right) = 0,$$
it follows that $ x_5 x_6 < 0 $ , which leads to a contradiction.
 \\
These two contradictions prove that $x_5 x_6 > 0$ and $y_5y_6>0$ cannot be true at the same time for a pyramidal central configuration $q$.\\

In the following, it suffices to consider the remaining three possible cases.

{\bf Proof of possible case (ii):}
For this case, we assume that $x_5 x_6>0$ and $y_5=y_6=0$.
Without loss of generality, assume that $x_5>0$ and $x_6>0$ and $y_5=y_6=0$. Then $D_{51}=D_{53}=({x_5}^{2}+a^2+{z_5}^{2})^{1/2}$ and $D_{61}=D_{63}= ({x_6}^{2}+a^2+{z_6}^{2})^{1/2}$.

 We just need to prove $m_3 =m_1$. If not, suppose $m_3\not=m_1$. From $$f_{y_1}+f_{y_3}=\frac{(m_1-m_3)(8\lambda a^3-M)}{4a^2M}=0,$$ and
 $$f_{y_2}=f_{y_4}=(m_1-m_3)a\left(\frac{\lambda}{M}-\frac{1}{(a^2+1)^{3/2}}\right)=0,$$
 we have $a=\sqrt{3}/3$ and $\lambda=3\sqrt{3}M/8$. Then, from
 $$f_{y_5}=(m_1-m_3)a\left(\frac{\lambda}{M}-\frac{1}{(a^2+{x_5}^{2}+{z_5}^{2})^{3/2}}\right)=0$$
 and $$f_{y_6}=(m_1-m_3)a\left(\frac{\lambda}{M}-\frac{1}{(a^2+{x_6}^{2}+{z_6}^{2})^{3/2}}\right)=0,$$
 we obtain ${x_5}^{2}+{z_5}^{2}=1$ and ${x_6}^{2}+{z_6}^{2}=1$. Then, $f_{y_i}=0$ holds for all $i=1,2,\cdots 6$, $f_{z_i}=0$ and $f_{x_i}=0$ hold for all $i=1,3$. Under the given assumptions that $z_5 > 0$ and $z_6 < 0$, we have $z_5 = \sqrt{1 - {x_5}^2}$ and $z_6 = -\sqrt{1 - {x_6}^2}$, respectively.
 From
\begin{eqnarray*}
f_{z_2}&=& m_5z_5 \left(\frac{1}{({x_5}^{2}+{z_5}^2-2x_5+1)^{3/2}}-\frac{3\sqrt{3}}{8}\right) + m_6z_6 \left(\frac{1}{({x_6}^{2}+{z_6}^{2}-2x_6+1)^{3/2}}-\frac{3\sqrt{3}}{8}\right)\\
 &=&m_5z_5 \left(\frac{1}{(2-2x_5)^{3/2}}-\frac{3\sqrt{3}}{8}\right) + m_6z_6 \left(\frac{1}{(2-2x_6)^{3/2}}-\frac{3\sqrt{3}}{8}\right) =0,
 \end{eqnarray*}
and
\begin{eqnarray*}
 f_{z_4}&=& m_5z_5 \left(\frac{1}{({x_5}^{2}+{z_5}^2+2x_5+1)^{3/2}}-\frac{3\sqrt{3}}{8}\right) + m_6z_6 \left(\frac{1}{({x_6}^{2}+{z_6}^{2}+2x_6+1)^{3/2}}-\frac{3\sqrt{3}}{8}\right)\\
 &=&m_5z_5 \left(\frac{1}{(2+2x_5)^{3/2}}-\frac{3\sqrt{3}}{8}\right) + m_6z_6 \left(\frac{1}{(2+2x_6)^{3/2}}-\frac{3\sqrt{3}}{8}\right) =0,
 \end{eqnarray*}
 one deduce that
 $$\frac{\left(\displaystyle\frac{1}{(2-2x_5)^{3/2}}-\frac{3\sqrt{3}}{8}\right) }{\left(\displaystyle\frac{1}{(2+2x_5)^{3/2}}-\frac{3\sqrt{3}}{8}\right) } = \frac{\left(\displaystyle\frac{1}{(2-2x_6)^{3/2}}-\frac{3\sqrt{3}}{8}\right) }{\left(\displaystyle\frac{1}{(2+2x_6)^{3/2}}-\frac{3\sqrt{3}}{8}\right) } $$ for some $0<x_5<1$ and $0<x_6<1.$ Let $$g^-(u)=\left(\frac{1}{(2-2u)^{3/2}}-\frac{3\sqrt{3}}{8}\right)$$ and $$g^+(u)=\left(\displaystyle\frac{1}{(2+2u)^{3/2}}-\frac{3\sqrt{3}}{8}\right).$$
 If $x_5$ and $x_6$ satisfy $f_{z_2}=0$ and $f_{z_4}=0$, $g^-(x_5)$ and $g^-(x_6)$ must have same sign or zero, and $g^+(x_5)$ and $g^+(x_6)$ must have same sign or zero. Then, $x_5$ and $x_6$ must be in the same interval between  $(0, 1/3)$, and $(1/3,1)$. Noting that $g^-(1/3)=0$ and $g^+(-1/3)=0$. We can verify that $g^-(u)/g^+(u)$ decreases in each interval. Therefore, from $f_{z_2}=0$ and $f_{z_4}=0$, one obtains $x_5=x_6$, then $z_5=-z_6$, and then $m_5=m_6$. Then, by $f_{x_2}=0$, we have
$-m_4/4+3\sqrt{3}m_4/4=0$, that is, $m_4=0$. This contradicts $m_4>0$. This contradiction shows that $m_3=m_1.$ \\

% We have $m_1=m_3$. Then $f_{y_i}=0$ for all $i=2,4,5,6$  and $$f_{y_1}=-f_{y_3}=\frac{m_1}{4a^3}+ \frac{m_2 +m_4}{(a^2 + 1)^{3/2}} +\frac{m_5}{(a^2 + {x_5}^{2} + {z_5}^{2})^{3/2}}+\frac{m_6}{(a^2 + {x_6}^{2} + {z_6}^{2})^{3/2}} -\lambda=0.b$$.
% $fx_1=fx_3, fz_1=fz_3$

{\bf Proof of possible case (iii):} This case can be proved in a similar way to {\bf case (ii)}. We omit its proof here. \\

{\bf Proof of possible case (iv):} Assume that $x_5=x_6=y_5=y_6=0.$ From $f_{x_1}=0$, $f_{x_5}=0,$ and $f_{x_6}=0,$ we have $$(m_2-m_4)\left(\frac{1}{(a^2+1)^{3/2}} -\frac{\lambda}{M}\right)=0,$$
$$(m_2-m_4)\left(\frac{1}{({z_5}^{2}+1)^{3/2}} -\frac{\lambda}{M}\right)=0,$$
and $$(m_2-m_4)\left(\frac{1}{({z_6}^{2}+1)^{3/2}} -\frac{\lambda}{M}\right)=0.$$
If $m_4\not=m_2$, then we derive a contraction that $m_5=m_1=0$ below.

Assume that $m_4\not=m_2$, then $\lambda/M=1/(a^2+1)^{3/2}$ and $a^2={z_5}^{2}={z_6}^{2}.$ From $f_{x_4}=0$, we have
$$f_{x_4}=\frac{m_1}{\left(a^{2}+1\right)^{\frac{3}{2}}}+\frac{m_2}{4}+\frac{m_3}{\left(a^{2}+1\right)^{\frac{3}{2}}}+\frac{m_5}{\left({z_5}^{2}+1\right)^{\frac{3}{2}}}+\frac{m_6}{\left({z_6}^{2}+1\right)^{\frac{3}{2}}}-\lambda \left(\frac{m_1+2m_2+m_3+m_5+m_6}{M}\right)=0,
$$
which is equivalent to $$\frac{m_2}{4}-\frac{2m_2}{(a^2+1)^{3/2}}=0.$$
Since $m_2>0$, it follows that $a^2={z_5}^{2}={z_6}^{2}=3,$ $\lambda=M/8=(m_1+m_2+m_3+m_4+m_5+m_6)/8.$ Substituting these into equations of central configuration, we have $f_{z_1}= \left(m_5 -m_6 \right) \left(2 \sqrt{2}-3 \sqrt{3}\right)/24=0$, which implies that $m_6=m_5$. Similarly, $f_{y_5}=\left(m_1 -m_3 \right) \left(3 \sqrt{3}-2 \sqrt{2}\right)/24=0$ implies $m_3=m_1$. Then, $f_{y_1}=0$ and $f_{z_5}=0$ become
$$ \frac{\left(-3 m_1 -3 m_5 \right) \sqrt{3}}{12}+\frac{m_5 \sqrt{2}}{6}+\frac{m_1}{12}=0 \ \  {\rm and} \ \ \frac{\left(3 m_1 +3 m_5 \right) \sqrt{3}}{12}-\frac{m_1 \sqrt{2}}{6}-\frac{m_5}{12}=0,$$
which imply $m_5=m_1=0$. This contradicts the assumption that $m_5>0.$ Therefore, $m_4=m_2.$
Moreover, a similar argument shows that $m_3=m_1.$  $\Box$

\section{Proof of Theorem \ref{thm2.2}}\label{sec4}\setcounter{equation}{0}
In this whole section, we assume that \( q_5 \) and \( q_6 \) lie on the symmetric \( z \)-axis, i.e. \( x_5 = y_5 = x_6 = y_6 = 0 \). Assume that \( q \) is a pyramidal central configuration for the mass vector \( m = (m_1, m_2, \cdots, m_6) \). \\

{\bf Proof of Theorem \ref{thm2.2} (case 1):}
By Theorem \ref{thm2.1} case 3, $m_3=m_1$ and $m_4=m_2$. Suppose $z_6=-z_5.$ We only need to prove that $m_6=m_5.$

Substituting all the assumptions, after simplification, the 18 equations of central configurations reduce to six equations: $$f_{x_2}=0, f_{y_1}=0, f_{z_1}=0, f_{z_2}=0, f_{z_5}=0, f_{z_6}=0.$$

If $a\not=1,$ from $z_5>0$ and  $$f_{z_1}-f_{z_2}= (m_5-m_6)z_5\left(\frac{1}{(a^2+{z_5}^{2})^{3/2}}-\frac{1}{({z_5}^{2}+1)^{3/2}}\right) =0,$$
we have $m_6=m_5.$

If $a=1,$ we have $m_2=m_1$ because $f_{x_2}+f_{y_1}=(2\sqrt{2}-1)(m_2-m_1)/4=0$. By direct computation,
$$f_{z_5}+f_{z_6}=-\frac{\left(m_5-m_6 \right) \left(8 \lambda {z_5}^{3}-(4 m_1 +m_5 +m_6) \right)}{4 {z_5}^{2} \left(4 m_1+m_5 +m_6 \right)}=0.$$
Then if $m_6\not=m_5,$ $\lambda= (4 m_1 +m_5 +m_6)/8  {z_5}^{3}.$ $f_{z_5}=0$ becomes $f_{z_5}=m_1[({z_5}^{2}+1)^{3/2}-8 {z_5}^3 ]/[2({z_5}^{2}+1)^{3/2} {z_5}^{2}]=0$, which implies that $z_5=\sqrt{3}/3. $ Thus, $f_{x_2}=-m_1 \left(2 \sqrt{2}+1-6 \sqrt{3}\right)/4=0$, which is impossible because $m_1>0.$ Therefore $m_6=m_5.$ \\
We complete the proof for case 1 of Theorem \ref{thm2.2}.\\

{\bf Proof of Theorem \ref{thm2.2} (case 2):}
By Theorem \ref{thm2.1} case 3, $m_3=m_1$ and $m_4=m_2$.

Substituting $x_5=x_6=y_5=y_6=0$, $m_3=m_1$ and $m_4=m_2$ into the equations of central configuration (\ref{1.2}), eighteen equations are reduced to six equations below.
\begin{eqnarray}\label{4.1}
\left\{
  \begin{array}{llllll}
f_{x_2}= \displaystyle-\frac{2 {m_1}}{\left(a^{2}+1\right)^{\frac{3}{2}}}-\frac{{m_2}}{4}-\frac{{m_5}}{\left({z_5}^{2}+1\right)^{\frac{3}{2}}}-\frac{{m_6}}{\left({z_6}^{2}+1\right)^{\frac{3}{2}}}+{\lambda},\\[0.2cm]
 f_{y_1}= \displaystyle\frac{2 {m_2} a}{\left(a^{2}+1\right)^{\frac{3}{2}}}+\frac{{m_1} a}{4 \left(a^{2}\right)^{\frac{3}{2}}}+\frac{{m_5} a}{\left(a^{2}+{z_5}^{2}\right)^{\frac{3}{2}}}+\frac{{m_6} a}{\left(a^{2}+{z_6}^{2}\right)^{\frac{3}{2}}}-{\lambda} a, \\[0.2cm]
f_{z_1}=\displaystyle\frac{{m_5} {z_5}}{\left(a^{2}+{z_5}^{2}\right)^{\frac{3}{2}}}+\frac{{m_6} {z_6}}{\left(a^{2}+{z_6}^{2}\right)^{\frac{3}{2}}}-\frac{{\lambda} \left({m_5} {z_5} +{m_6} {z_6} \right)}{2 {m_1} +2 {m_2} +{m_5} +{m_6}},\\[0.2cm]
f_{z_2}=\displaystyle\frac{{m_5} {z_5}}{\left({z_5}^{2}+1\right)^{\frac{3}{2}}}+\frac{{m_6} {z_6}}{\left({z_6}^{2}+1\right)^{\frac{3}{2}}}-\frac{{\lambda} \left({m_5} {z_5} +{m_6} {z_6} \right)}{2 {m_1} +2 {m_2} +{m_5} +{m_6}}, \\[0.2cm]
f_{z_5}=\displaystyle-\frac{2 {m_1} {z_5}}{\left(a^{2}+{z_5}^{2}\right)^{\frac{3}{2}}}-\frac{2 {m_2} {z_5}}{\left({z_5}^{2}+1\right)^{\frac{3}{2}}}+\frac{{m_6} \left({z_6} -{z_5} \right)}{\left(\left({z_6} -{z_5} \right)^{2}\right)^{\frac{3}{2}}}+{\lambda} \left({z_5} -\frac{{m_5} {z_5} +{m_6} {z_6}}{2 {m_1} +2 {m_2} +{m_5} +{m_6}}\right),\\[0.2cm]
f_{z_6}=\displaystyle-\frac{2 {m_1} {z_6}}{\left(a^{2}+{z_6}^{2}\right)^{\frac{3}{2}}}-\frac{2 {m_2} {z_6}}{\left({z_6}^{2}+1\right)^{\frac{3}{2}}}+\frac{{m_5} \left(-{z_6} +{z_5} \right)}{\left(\left({z_6} -{z_5} \right)^{2}\right)^{\frac{3}{2}}}+{\lambda} \left({z_6} -\frac{{m_5} {z_5} +{m_6} {z_6}}{2 {m_1} +2 {m_2} +{m_5} +{m_6}}\right).
  \end{array}
\right.
\end{eqnarray}

We have seven variables $m_1, m_2, m_5, m_6, a, z_5, z_6,$ and $ \lambda$.
Suppose $z_6\not=-z_5.$ We only need to prove that $a=1$ and $m_2=m_1$. We first prove the following lemmas.

\begin{lemma} \label{lem4.1}
If the six masses form a central configuration with $a > 1$ and $x_5 = x_6 = y_5 = y_6 = 0$, then:
\begin{itemize}
    \item When $z_6 > -z_5$, we have $m_5 z_5 + m_6 z_6 > 0$;
    \item Conversely, when $z_6 < -z_5$, we have $m_5 z_5 + m_6 z_6 < 0$.
\end{itemize}
\end{lemma}
\textbf{Proof.} Since $a>1$ and $x_{5}=x_{6}=y_{5}=y_{6}=0$, it follows that $({z_{5}}^{2}+a^{2})^{3/2}\neq ({z_5}^{2}+1)^{3/2}$. For the central configuration, we use the definition $M=\sum_{1\leq j\leq 6}m_{j}$ and the conditions $f_{z_1}=0$ and $f_{z_2}=0$ from equations (\ref{4.1}) to find that
\begin{equation}\label{4.2}
\frac{m_5 z_5}{(a^{2}+{z_{5}}^{2})^{\frac{3}{2}}}+\frac{m_6 z_6}{(a^{2}+z^{2}_{6})^{\frac{3}{2}}}=\lambda z_0=\lambda\frac{m_{5}z_{5}+m_{6}z_{6}}{M}
\end{equation}
and
\begin{equation}\label{4.3}
\frac{m_5 z_5}{({z_5}^{2}+1)^{\frac{3}{2}}}+\frac{m_6 z_6}{({z_6}^{2}+1)^{\frac{3}{2}}}=\lambda z_0=\lambda\frac{m_{5}z_{5}+m_{6}z_{6}}{M}.
\end{equation}
Using $({z_5}^{2}+a^{2})^{3/2}\neq ({z_5}^{2}+1)^{3/2}$, $m_{6}z_{6}<0$, equations (\ref{4.2}) and (\ref{4.3}), one computes that
\begin{eqnarray}\label{4.4}
\frac{m_5 z_5}{m_6z_6}=\frac{\displaystyle\frac{1}{(1 +{z_6}^{2})^{\frac{3}{2}}}-\frac{1}{(a^{2}+{z_6}^{2})^{\frac{3}{2}}}}{\displaystyle\frac{1}{(a^{2}+{z_5}^{2})^{\frac{3}{2}}}-\frac{1}{(1 +{z_5}^{2})^{\frac{3}{2}}}}.
\end{eqnarray}
Moreover, let
\begin{eqnarray}\label{4.5}
f(u)=\frac{1}{(1+u)^{\frac{3}{2}}}-\frac{1}{(a^{2}+u)^{\frac{3}{2}}}, \ \   {\rm where} \ \ u>0.
\end{eqnarray}
By directly differentiating $f(u)$ and using the condition $a > 1$, we have the derivative
\begin{eqnarray}\label{4.6}
f^{\prime}(u)=\frac{3}{2}\left[\frac{1}{(a^{2}+u)^{\frac{5}{2}}}-\frac{1}{(1+u)^{\frac{5}{2}}}\right]<0.
\end{eqnarray}
Thus, $f(u)$ decreases in $u$.
Therefore, under the assumption that $0 > z_6 > -z_5$ and the conditions $x_5 = x_6 = y_5 = y_6 = 0$ and $a > 1$, we deduce that
\begin{eqnarray}\label{4.7}
\frac{1}{(1+{z_6}^{2})^\frac{3}{2}}-\frac{1}{(a^{2}+{z_6}^{2})^\frac{3}{2}}
>\frac{1}{(1+{z_5}^{2})^\frac{3}{2}}-\frac{1}{(a^{2}+{z_5}^{2})^\frac{3}{2}}>0.
\end{eqnarray}
Inserting inequality (\ref{4.7}), $z_{5}>0$ and $z_{6}<0$ into equation (\ref{4.4}), we have $m_{5}z_{5}>-m_{6}z_{6}>0$, which implies that $m_{5}z_{5}+m_{6}z_{6}>0$.\\
Similarly, under the condition $z_{6} < -z_{5} < 0$, we derive the inequality:
\begin{eqnarray*}
0<\frac{1}{(1+{z_6}^{2})^\frac{3}{2}}-\frac{1}{(a^{2}+{z_6}^{2})^\frac{3}{2}}
<\frac{1}{(1+{z_5}^{2})^\frac{3}{2}}-\frac{1}{(a^{2}+{z_5}^{2})^\frac{3}{2}}.
\end{eqnarray*}
Then, combining $z_{5}>0$, $z_{6}<0$ with equation (\ref{4.4}), we have $m_{5}z_{5}+m_{6}z_{6}<0$. $\Box$

Based on Lemma \ref{lem4.1}, we obtain
\begin{lemma} \label{lem4.2}
Consider the case where $z_{5} \neq -z_{6}$ with $x_{5} = x_{6} = y_{5} = y_{6} = 0$ and $a > 1$. If the six masses form a central configuration, then $1/|{z_5}-{z_6}|^{3}<\lambda/M <\min\Big\{1/(a^{2}+{z_5}^{2})^{3/2}, \,\, 1/(a^{2}+{z_6}^{2})^{3/2}\Big\}$.
\end{lemma}

{\bf Proof.}  Firstly, we use Lemma \ref{lem4.1} to deduce that $\lambda/M<\min\Big\{1/(a^{2}+{z_5}^{2})^{3/2},   1/(a^{2}+{z_6}^{2})^{3/2}\Big\}$.

In fact, equation (\ref{4.2}) can be rewritten as
 \begin{eqnarray}\label{4.8}
m_5 z_5\left[\frac{1}{(a^{2}+{z_5}^{2})^{\frac{3}{2}}}-\frac{\lambda}{M}\right]=m_6 z_6\left[\frac{\lambda}{M}-\frac{1}{(a^{2}+{z_6}^{2})^{\frac{3}{2}}}\right],
\end{eqnarray}
which implies that either both $\lambda/M=1/(a^{2}+{z_5}^{2})^{3/2}$ and  $\lambda/M=1/(a^{2}+{z_6}^{2})^{3/2}$ hold, or both $\lambda/M\neq 1/(a^{2}+{z_5}^{2})^{3/2}$ and $\lambda/M\neq 1/(a^{2}+{z_6}^{2})^{3/2}$ hold. If both $\lambda/M=1/(a^{2}+{z_5}^{2})^{3/2}$ and $\lambda/M=1/(a^{2}+{z_6}^{2})^{3/2}$ hold, then $a^{2}+{z_5}^{2}=a^{2}+{z_6}^{2}$. From $z_{5}\cdot z_{6}<0$, we have $z_5=-z_6$, which contradicts $z_{5}\neq -z_{6}$. Hence, we obtain $\lambda/M\neq1/(a^{2}+{z_5}^{2})^{3/2}$ and $\lambda/M\neq1/(a^{2}+{z_6}^{2})^{3/2}$.
By equation (\ref{4.8}), we get
 \begin{eqnarray}\label{4.9}
\frac{m_5 z_5}{m_6 z_6}=\frac{\displaystyle\frac{\lambda}{M}-\frac{1}{(a^{2}+{z_6}^{2})^{\frac{3}{2}}}}{\displaystyle\frac{1}{(a^{2}+{z_5}^{2})^{\frac{3}{2}}}-\frac{\lambda}{M}}.
\end{eqnarray}
 Next, for the case of $z_{5}\neq -z_{6}$, we will prove that $\lambda/M<1/(a^{2}+{z_5}^{2})^{3/2}$ by contradiction as follows.

 From $\lambda/M\neq1/(a^{2}+{z_5}^{2})^{3/2}$, here we assume that $\lambda/M>1/(a^{2}+{z_5}^{2})^{3/2}$. Then, by $m_{5}z_{5}/(m_{6}z_{6})<0$ and equation (\ref{4.9}), we have $\lambda/M>1/(a^{2}+{z_6}^{2})^{3/2}$.

Under the condition that $0 < -z_6 < z_5$, and given that $a > 1$ and $x_5 = x_6 = y_5 = y_6 = 0$, it follows from Lemma \ref{lem4.1} that $m_5 z_5 + m_6 z_6 > 0$. Thus, from $m_{6}z_{6}<0$, we have $[m_{5}z_{5}]/[m_{6}z_{6}]<-1$. Noticed that $\lambda/M>1/(a^{2}+{z_5}^{2})^{3/2}$, then by $\lambda/M>1/(a^{2}+{z_6}^{2})^{3/2}$ and equation (\ref{4.9}), we have $1/(a^{2}+{z_6}^{2})^{3/2}<1/(a^{2}+{z_5}^{2})^{3/2}$.
On the other hand, since ${z_6}^{2}< {z_5}^{2}$, it is easy to verify that $1/(a^{2}+{z_6}^{2})^{3/2}>1/(a^{2}+{z_5}^{2})^{3/2}$. This produces a contradiction.

In the case where $-z_6 > z_5 > 0$, following an argument similar to that for the case $0<-z_6 < z_5$, Lemma \ref{lem4.1} implies that $m_5 z_5 + m_6 z_6 < 0$. Then, it follows from $m_{6}z_{6}<0$ that $-1<[m_{5}z_{5}]/[m_{6}z_{6}]<0$. Thus, inserting $\lambda/M>1/(a^{2}+{z_5}^{2})^{3/2}$ and $\lambda/M>1/(a^{2}+{z_6}^{2})^{3/2}$ into equation (\ref{4.9}), we obtain $1/(a^{2}+{z_6}^{2})^{3/2}>1/(a^{2}+{z_5}^{2})^{3/2}$. However, this result is impossible because the condition $-z_{6} > z_{5}>0$ directly implies the opposite inequality $1/(a^{2}+z_6^{2})^{3/2} < 1/(a^{2}+z_5^{2})^{3/2}$, thus creating a contradiction.
So, we conclude that $\lambda/M<1/(a^{2}+{z_5}^{2})^{3/2}$. Then, it follows from equation (\ref{4.8}) that $\lambda/M<1/(a^{2}+{z_6}^{2})^{3/2}$ and $\lambda/M<1/(a^{2}+{z_5}^{2})^{3/2}$, i.e., $\lambda/M<\min\Big\{1/(a^{2}+{z_5}^{2})^{3/2},   1/(a^{2}+{z_6}^{2})^{3/2}\Big\}$.

Secondly, we use $\lambda/M<\min\Big\{1/(a^{2}+{z_5}^{2})^{3/2}, 1/(a^{2}+{z_6}^{2})^{3/2}\Big\}$ to prove that $ 1/|{z_5}-{z_6}|^{3}<\lambda/M$. In fact, with the aid of the fifth and sixth equations of (\ref{4.1}), we have
\begin{eqnarray}\label{4.10}
\left\{
  \begin{array}{ll}
{\lambda}{z_5}=\displaystyle\frac{2 {m_1} {z_5}}{\left(a^{2}+{z_5}^{2}\right)^{\frac{3}{2}}}+\frac{2 {m_2} {z_5}}{\left({z_5}^{2}+1\right)^{\frac{3}{2}}}+\frac{{m_6} \left(-{z_6} +{z_5} \right)}{\left(\left({z_5} -{z_6} \right)^{2}\right)^{\frac{3}{2}}}+ \frac{\lambda({m_5} {z_5} +{m_6} {z_6})}{2 {m_1} +2 {m_2} +{m_5} +{m_6}}, \\[0.2cm]
-{\lambda} {z_6}=-\displaystyle\frac{2 {m_1} {z_6}}{\left(a^{2}+{z_6}^{2}\right)^{\frac{3}{2}}}-\frac{2 {m_2} {z_6}}{\left({z_6}^{2}+1\right)^{\frac{3}{2}}}+\frac{{m_5} \left(-{z_6} +{z_5} \right)}{\left(\left({z_6} -{z_5} \right)^{2}\right)^{\frac{3}{2}}}- \frac{\lambda({m_5} {z_5} +{m_6} {z_6})}{2 {m_1} +2 {m_2} +{m_5} +{m_6}},
  \end{array}
\right.
\end{eqnarray}
which implies that
\begin{eqnarray}\label{4.11}
{\lambda}(z_{5}-z_{6})=\displaystyle\frac{2 {m_1} {z_5}}{\left(a^{2}+{z_5}^{2}\right)^{\frac{3}{2}}}-\frac{2 {m_1} {z_6}}{\left(a^{2}+{z_6}^{2}\right)^{\frac{3}{2}}}+\frac{2 {m_2} {z_5}}{\left({z_5}^{2}+1\right)^{\frac{3}{2}}}-\frac{2 {m_2} {z_6}}{\left({z_6}^{2}+1\right)^{\frac{3}{2}}}+\frac{\left({m_5}+{m_6}\right)\left({z_5} -{z_6}\right)}{|{z_5}-{z_6}|^{3}}.
\end{eqnarray}

When ${z_{5}}^{2}>{z_{6}}^{2}$: inserting $a>1$, $z_6<0<z_5$ and $\lambda/M<1/(a^{2}+{z_5}^{2})^{3/2}$ into equation (\ref{4.11}), one computes that
\begin{eqnarray}\label{4.12}
\frac{M(z_5-z_6)}{(a^{2}+{z_5}^{2})^{3/2}}>{\lambda}{(z_5-z_6)}
>\displaystyle\frac{{2} ({z_5}-{z_6})}{\left(a^{2}+{z_5}^{2}\right)^{\frac{3}{2}}}m_1+\frac{{2} ({z_5}-{z_6})}{\left(a^{2}+{z_5}^{2}\right)^{\frac{3}{2}}}m_2+\frac{({m_5}+{m_6}) \left({z_5}-{z_6} \right)}{|{z_5}-{z_6}|^{3}}.
\end{eqnarray}
On the other hand, noticed that $x_5 = x_6 = y_5=y_6=0$, then from Theorem \ref{thm2.1} case 3, we have $m_3=m_1$ and $m_4=m_2$. Hence, $M=\sum_{1\leq j\leq 6}m_{j}=2m_{1}+2m_{2}+m_{5}+m_{6}$. Then, using inequality (\ref{4.12}), we have
\begin{eqnarray*}
\frac{({m_5}+{m_6})(z_5-z_6)}{(a^{2}+{z_5}^{2})^{3/2}}
>\displaystyle\frac{({m_5}+{m_6})\left({z_5}-{z_6} \right)}{|{z_5}-{z_6}|^{3}},
\end{eqnarray*}
which implies that
$|{z_5}-{z_6}|^{3}>\left(a^{2}+{z_5}^{2}\right)^{3/2}$. Inserting $|{z_5}-{z_6}|^{3}>\left(a^{2}+{z_5}^{2}\right)^{3/2}$ into inequality (\ref{4.12}), we have
\begin{eqnarray}\label{4.13}
{\lambda}{(z_5-z_6)}
>\displaystyle\frac{{2} ({z_5}-{z_6})}{|{z_5}-{z_6}|^{3}}m_1+\frac{{2} ({z_5}-{z_6})}{|{z_5}-{z_6}|^{3}}m_2+\frac{({m_5}+{m_6}) \left({z_5}-{z_6} \right)}{|{z_5}-{z_6}|^{3}},
\end{eqnarray}
and it implies that when ${z_{5}}^{2}>{z_{6}}^{2}$, we have $ 1/|{z_5}-{z_6}|^{3}<\lambda/M$.

When ${z_5}^{2}<{z_6}^{2}$: noticed that $\lambda/M<\min\Big\{1/(a^{2}+{z_5}^{2})^{3/2}, 1/(a^{2}+{z_6}^{2})^{3/2}\Big\}$, then
inserting $a>1$, $z_6<0<z_5$ and $\lambda/M<1/(a^{2}+{z_6}^{2})^{3/2}$ into equation (\ref{4.11}), we obtain
\begin{eqnarray}\label{4.14}
\frac{M(z_5-z_6)}{(a^{2}+{z_6}^{2})^{3/2}}>{\lambda}{(z_5-z_6)}
>\displaystyle\frac{{2} ({z_5}-{z_6})}{\left(a^{2}+{z_6}^{2}\right)^{\frac{3}{2}}}m_1+\frac{{2} ({z_5}-{z_6})}{\left(a^{2}+{z_6}^{2}\right)^{\frac{3}{2}}}m_2+\frac{({m_5}+{m_6}) \left({z_5}-{z_6} \right)}{\left(\left({z_6} -{z_5} \right)^{2}\right)^{\frac{3}{2}}},
\end{eqnarray}
and then using $m_3=m_1$ and $m_4=m_2$, it follows that $|{z_5}-{z_6}|^{3}>\left(a^{2}+{z_6}^{2}\right)^{3/2}$. Therefore, by inequality (\ref{4.14}), we have inequality (\ref{4.13}) holds for the case of ${z_5}^{2}<{z_6}^{2}$, too. Thus, when ${z_{5}}^{2}>{z_{6}}^{2}$, $ 1/|{z_5}-{z_6}|^{3}<\lambda/M$ holds.

Hence, in the case where $z_{5}\neq -z_{6}$, when $x_{5}=x_{6}=y_{5}=y_{6}=0$ and $a>1$: we conclude that $1/|{z_5}-{z_6}|^{3}<\lambda/M <\min\Big\{1/(a^{2}+{z_5}^{2})^{3/2}, \,\, 1/(a^{2}+{z_6}^{2})^{3/2}\Big\}$.
 $\Box$ \\

In what follows, we firstly use Lemma \ref{lem4.2} to prove that $a=1$ by contradiction. Assuming that $a\neq1$ and recalling that $x_{5}=x_{6}=y_{5}=y_{6}=0$, we note that by the symmetry of the central configuration, it suffices to consider the case $a>1$.

Recall that $a>1$ and $z_6<0$. By Lemma \ref{lem4.2}, we obtain $1/|{z_5}-{z_6}|^{3}<1/\left(a^{2}+{z_6}^{2}\right)^{3/2}$. Thus,
\begin{eqnarray}\label{4.15}
\displaystyle\frac{{m_6} {z_6}} {|{z_5}-{z_6}|^{3}}-\frac{{m_6} {z_6}}{\left(a^{2}+{z_6}^{2}\right)^{\frac{3}{2}}}>0.
\end{eqnarray}
Moreover, inserting $f_{z_{1}}=0$ into the first equation of (\ref{4.10}), it follows that
\begin{eqnarray}\label{4.16}
\displaystyle{\lambda} {z_5}=\frac{2 {m_1} {z_5}}{\left(a^{2}+{z_5}^{2}\right)^{\frac{3}{2}}}+\frac{2 {m_2} {z_5}}{\left({z_5}^{2}+1\right)^{\frac{3}{2}}}-\left[\frac{{m_6} {z_6}} {|{z_5}-{z_6}|^{3}}-\frac{{m_6} {z_6}}{\left(a^{2}+{z_6}^{2}\right)^{\frac{3}{2}}}\right] +\frac{{m_5} {z_5}}{\left(a^{2}+{z_5}^{2}\right)^{\frac{3}{2}}}+\frac{{m_6} {z_5}} {|{z_5}-{z_6}|^{3}}.
\end{eqnarray}
Combining inequality (\ref{4.15}) with equation (\ref{4.16}), we have
\begin{eqnarray*}
\displaystyle{\lambda} {z_5}<\frac{2 {m_1} {z_5}}{\left(a^{2}+{z_5}^{2}\right)^{\frac{3}{2}}}+\frac{2 {m_2} {z_5}}{\left({z_5}^{2}+1\right)^{\frac{3}{2}}}+\frac{{m_5} {z_5}}{\left(a^{2}+{z_5}^{2}\right)^{\frac{3}{2}}}+\frac{{m_6} {z_5}} {|{z_5}-{z_6}|^{3}}.
\end{eqnarray*}
Therefore, it follows from $a>1$ and $z_{5}>0$ that
\begin{eqnarray}\label{4.17}
\displaystyle{\lambda}<\frac{2 {m_1}+2 {m_2}+m_{5} }{\left({z_5}^{2}+1\right)^{\frac{3}{2}}}+\frac{{m_6}} {|{z_5}-{z_6}|^{3}}.
\end{eqnarray}
Let
\begin{eqnarray}\label{4.18}
\left\{
  \begin{array}{ll}
g(u)=m_5 z_5\left[\frac{M}{\left({z_5}^{2}+1\right)^{\frac{3}{2}}}-u\right]+m_6 z_6\left[\frac{M}{\left({z_6}^{2}+1\right)^{\frac{3}{2}}}-u\right], \ \ \  {\rm where} \ \ u>0, \\[0.4cm]
\widetilde{g}(v)=1+\displaystyle\frac{\frac{1}{\left({z_6}^{2}+1\right)^{\frac{3}{2}}}-\frac{1}{\left({z_5}^{2}+1\right)^{\frac{3}{2}}}}{\frac{1}{\left({z_5}^{2}+1\right)^{\frac{3}{2}}}-v}, \ \ \  {\rm where} \ \ v<\frac{1}{\left({z_5}^{2}+1\right)^{\frac{3}{2}}}.
  \end{array}
\right.
\end{eqnarray}
Noticed that Lemma \ref{lem4.1}, we have
\begin{eqnarray}\label{4.19}
\left\{
  \begin{array}{ll}
m_5 z_5 + m_6 z_6 > 0, \ \ {\rm when}  \ \ z_{5}>-z_{6},\\[0.4cm]
m_5 z_5 + m_6 z_6 < 0, \ \ {\rm when} \ \ z_{5}<-z_{6}.
  \end{array}
\right.
\end{eqnarray}
By equations (\ref{4.18}), inequalities (\ref{4.19}), and direct computation, one computes that the derivative
\begin{eqnarray*}
g^{\prime}(u)=-(m_5 z_5 + m_6 z_6)=\left\{
  \begin{array}{ll}
< 0, \ \ {\rm when}  \ \ z_{5}>-z_{6},\\[0.4cm]
> 0, \ \ {\rm when} \ \ z_{5}<-z_{6},
  \end{array}
\right.
\end{eqnarray*}
and the derivative
\begin{eqnarray*}
\widetilde{g}^{\prime}(v)=\left[\frac{1}{\left({z_6}^{2}+1\right)^{\frac{3}{2}}}-\frac{1}{\left({z_5}^{2}+1\right)^{\frac{3}{2}}}\right]
\displaystyle\frac{1}{\left[\frac{1}{\left({z_5}^{2}+1\right)^{\frac{3}{2}}}-v\right]^{2}}, \ \ \  {\rm where} \ \ v<\frac{1}{\left({z_5}^{2}+1\right)^{\frac{3}{2}}}.
\end{eqnarray*}
Hence, when $z_{5}>-z_{6}$, $g(u)$ decreases in $u$ but $\widetilde{g}(v)$ increase in $v$; and when $z_{5}<-z_{6}$, $g(u)$ increases in $u$ but $\widetilde{g}(v)$ decrase in $v$. Moreover, using the first equation of equations (\ref{4.18}) and the fourth equation of equations (\ref{4.1}), we have $g(\lambda)=0$.

Next, we divide the following proof into two parts.

{\bf Part 1:} We show that when $z_{5}>-z_{6}$, there is a contradiction.

In fact, by $z_{5}>-z_{6}$, we see that $g(u)$ decreases in $u$. Combining $g(\lambda)=0$ and inequality (\ref{4.17}), when $z_{5}>-z_{6}$, we deduce that
\begin{eqnarray*}
g\left(\frac{2 {m_1}+2 {m_2}+m_{5} }{\left({z_5}^{2}+1\right)^{\frac{3}{2}}}+\frac{{m_6}} {|{z_5}-{z_6}|^{3}}\right)<0,
\end{eqnarray*}
and then it follows from the first equation of equations (\ref{4.18}) that
\begin{eqnarray*}
m_{5}z_{5}\left[\frac{m_{6}}{\left({z_5}^{2}+1\right)^{\frac{3}{2}}}-\frac{{m_6}} {|{z_5}-{z_6}|^{3}}\right]+m_{6}z_{6}\Big[\frac{m_{6}}{\left({z_6}^{2}+1\right)^{\frac{3}{2}}}-\frac{{m_6}} {|{z_5}-{z_6}|^{3}}\Big]<0.
\end{eqnarray*}
Note that Lemma \ref{lem4.2} implies $1/\left(a^{2}+{z_6}^{2}\right)^{3/2}-1/|{z_5}-{z_6}|^{3}>0$. Using the fact that $-m_{6}z_{6}>0$, we have
\begin{eqnarray}\label{4.20}
\frac{{m_{5}}{z_{5}}}{-{m_{6}}{z_{6}}}<\frac{\frac{1}{\left({z_6}^{2}+1\right)^{\frac{3}{2}}}-\frac{1} {|{z_5}-{z_6}|^{3}}}{\frac{1}{\left({z_5}^{2}+1\right)^{\frac{3}{2}}}-\frac{1} {|{z_5}-{z_6}|^{3}}}.
\end{eqnarray}
On the one hand, with the aid of $1/\left({z_5}^{2}+1\right)^{3/2}-\lambda/M>0$ and the fourth equation of equations (\ref{4.1}), one computes that
\begin{eqnarray}\label{4.21}
\displaystyle\frac{\frac{1}{\left({z_6}^{2}+1\right)^{\frac{3}{2}}}-\frac{\lambda}{M}}{\frac{1}{\left({z_5}^{2}+1\right)^{\frac{3}{2}}}-\frac{\lambda}{M}}
=\frac{{m_{5}}{z_{5}}}{-{m_{6}}{z_{6}}}.
\end{eqnarray}
Using inequality (\ref{4.20}) and equation (\ref{4.21}), we get
\begin{eqnarray}\label{4.22}
\displaystyle\frac{\frac{1}{\left({z_6}^{2}+1\right)^{\frac{3}{2}}}-\frac{\lambda}{M}}{\frac{1}{\left({z_5}^{2}+1\right)^{\frac{3}{2}}}-\frac{\lambda}{M}}
<\frac{\frac{1}{\left({z_6}^{2}+1\right)^{\frac{3}{2}}}-\frac{1} {|{z_5}-{z_6}|^{3}}}{\frac{1}{\left({z_5}^{2}+1\right)^{\frac{3}{2}}}-\frac{1} {|{z_5}-{z_6}|^{3}}}.
\end{eqnarray}
On the other hand, by Lemma \ref{lem4.2} and $a>1$, we have $1/|{z_5}-{z_6}|^{3}<\lambda/M <1/(a^{2}+{z_5}^{2})^{3/2}<1/(1+{z_5}^{2})^{3/2}$.
Recall that when $z_{5}>-z_{6}$, $\widetilde{g}(v)$ increase in $v$. Then, together with the second equation of equations (\ref{4.18}),
we have
\begin{eqnarray*}
\widetilde{g}(\frac{\lambda}{M})=\displaystyle\frac{\frac{1}{\left({z_6}^{2}+1\right)^{\frac{3}{2}}}-\frac{\lambda}{M}}{\frac{1}{\left({z_5}^{2}+1\right)^{\frac{3}{2}}}-\frac{\lambda}{M}}
>\widetilde{g}(|{z_5}-{z_6}|^{3})=\frac{\frac{1}{\left({z_6}^{2}+1\right)^{\frac{3}{2}}}-\frac{1} {|{z_5}-{z_6}|^{3}}}{\frac{1}{\left({z_5}^{2}+1\right)^{\frac{3}{2}}}-\frac{1} {|{z_5}-{z_6}|^{3}}},
\end{eqnarray*}
which contradicts inequality (\ref{4.22}). So, when $z_{5}>-z_{6}$, $a>1$ is impossible.

{\bf Part 2:} We show that when $z_{5}<-z_{6}$, there is a contradiction.

By $z_{5}<-z_{6}$, we see that $g(u)$ increase in $u$. Combining $g(\lambda)=0$ and inequality (\ref{4.17}), when $z_{5}<-z_{6}$, we deduce that
\begin{eqnarray*}
g\left(\frac{2 {m_1}+2 {m_2}+m_{5} }{\left({z_5}^{2}+1\right)^{\frac{3}{2}}}+\frac{{m_6}} {|{z_5}-{z_6}|^{3}}\right)>0,
\end{eqnarray*}
and then the first equation of equations (\ref{4.18}) implies that
\begin{eqnarray*}
m_{5}z_{5}\left[\frac{m_{6}}{\left({z_5}^{2}+1\right)^{\frac{3}{2}}}-\frac{{m_6}} {|{z_5}-{z_6}|^{3}}\right]+m_{6}z_{6}\Big[\frac{m_{6}}{\left({z_6}^{2}+1\right)^{\frac{3}{2}}}-\frac{{m_6}} {|{z_5}-{z_6}|^{3}}\Big]>0.
\end{eqnarray*}
Next, we deduce a contradiction by following a procedure similar to that used in the case $z_{5}>-z_{6}$, and the only difference lies in the fact that for $z_{5}<-z_{6}$, $\widetilde{g}(v)$ decrease in $v$. Here, we omit the details.
Thus, $a>1$ is also impossible when $z_{5}<-z_{6}$.

By now, for the central configuration, we conclude that if $z_{5}\neq-z_{6}$, then $a>1$ is impossible, and then we complete the proof of $a=1$.

Furthermore, inserting $a=1$ into the first two equations of equations (\ref{4.1}), we have
\begin{eqnarray*}
\left\{
  \begin{array}{ll}
{\lambda}= \displaystyle\frac{2 {m_1}}{2^{\frac{3}{2}}}+\frac{{m_2}}{4}+\frac{{m_5}}{\left({z_5}^{2}+1\right)^{\frac{3}{2}}}+\frac{{m_6}}{\left({z_6}^{2}+1\right)^{\frac{3}{2}}},\\[0.2cm]
 {\lambda}= \displaystyle\frac{2 {m_2} }{2^{\frac{3}{2}}}+\frac{{m_1} }{4}+\frac{{m_5} }{\left(1+{z_5}^{2}\right)^{\frac{3}{2}}}+\frac{{m_6} }{\left(1+{z_6}^{2}\right)^{\frac{3}{2}}},
  \end{array}
\right.
\end{eqnarray*}
which implies that $(1/\sqrt{2}-1/4)(m_{1}-m_{2})=0$, and then $m_{1}=m_{2}$. Combining Theorem \ref{thm2.1} case 3, we arrive at the conclusion that $m_{1}=m_{2}=m_{3}=m_{4}$ and $a=1$. $\Box$

\section{Proof of Theorem \ref{thm2.3}}\label{sec5}\setcounter{equation}{0}
Aussume that $z_6 = -z_5$, $m_3 = m_1$, $m_4 = m_2$, $m_6 = m_5$, and $x_5=x_6=y_5=y_6=0$. Then, the equations of the central configuration (\ref{1.2}) are equivalent to
\begin{eqnarray}\label{5.1}
\left\{
  \begin{array}{lll}
\displaystyle\frac{2 {m_1}}{\left(a^{2}+1\right)^{\frac{3}{2}}}+\frac{{m_2}}{4}+\frac{{2m_5}}{\left({z_5}^{2}+1\right)^{\frac{3}{2}}}={\lambda},\\[0.2cm]
\displaystyle\frac{2 {m_2}}{\left(a^{2}+1\right)^{\frac{3}{2}}}+\frac{{m_1}}{4 \left(a^{2}\right)^{\frac{3}{2}}}+\frac{{2m_5}}{\left(a^{2}+{z_5}^{2}\right)^{\frac{3}{2}}}={\lambda}, \\[0.2cm]
\displaystyle\frac{2 {m_1} }{\left(a^{2}+{z_5}^{2}\right)^{\frac{3}{2}}}+\frac{2 {m_2}}{\left({z_5}^{2}+1\right)^{\frac{3}{2}}}+\frac{{2m_5}}{\left(4{z_5} ^{2}\right)^{\frac{3}{2}}}={\lambda}.
  \end{array}
\right.
\end{eqnarray}
Define
$$
u=a^{2}, \ v=z_{5}^{2}, \ A_{1}=\frac{1}{(a^{2}+1)^{\frac{3}{2}}}, \ A_{2}=\frac{1}{(a^{2})^{\frac{3}{2}}}, \ A_{3}=\frac{1}{(z_{5}^{2}+1)^{\frac{3}{2}}}, \ A_{4}=\frac{1}{(z_{5}^{2})^{\frac{3}{2}}}, \ A_{5}=\frac{1}{(a^{2}+{z_5}^{2})^{\frac{3}{2}}}.
$$
Subtracting the first equation from the second equation and from the third in equations (\ref{5.1}) gives the equivalent system
\begin{eqnarray}\label{5.2}
\left\{
  \begin{array}{ll}
F(u,\,v)=\displaystyle8(m_{1}-m_{2})A_{1}+m_{2}-m_{1}A_{2}+8m_{5}(A_{3}-A_{5})=0,\\[0.2cm]
G(u,\,v)=\displaystyle8m_{2}A_{1}+m_{1}A_{2}+8m_{5}A_{5}-8m_{1}A_{5}-8m_{2}A_{3}-m_{5}A_{4}=0.
  \end{array}
\right.
\end{eqnarray}
Fix $u>0$. Consider $G(u,\,v)$ as a function of $z_{5}^{2}>0$. As $v\rightarrow0^{+}$, we have $A_{3}\rightarrow1$, $A_{4}\rightarrow+\infty$ and $A_{5}\rightarrow A_{2}$. Thus, when $v\rightarrow0$, $G(u,\,v)\rightarrow-\infty$. On the other hand, we compute that as $v\rightarrow+\infty$, $G(u,\,v)\rightarrow8m_{2}A_{1}+m_{1}A_{2}>0$. Using continuity, for each $u>0$, there exists $v=v_{0}(u)>0$ such that $G(u,\,v_{0}(u))=0$.

Next define $H(u):=F(u,v_{0}(u))$. As $u\rightarrow0^{+}$, we have $A_{1}\rightarrow1$, $A_{2}\rightarrow+\infty$, and from $G=0$, we must have $v_{0}(u)\rightarrow0^{+}$, hence $A_{3}\rightarrow1$, $A_{5}\rightarrow+\infty$, and $A_{2}\rightarrow+\infty$. Therefore, $H(u)\rightarrow-\infty$ as $u\rightarrow0^{+}$. Similarly, as $u\rightarrow+\infty$, a direct computation show that $A_{1}\rightarrow0$, $A_{2}\rightarrow0$ and $A_{5}\rightarrow0$. Thus $G=0$ and implies
$v_{0}(u)\rightarrow+\infty$, and then $A_{3},A_{4}\rightarrow0$. Hence, $H(u)\rightarrow m_{2}>0$ as $u\rightarrow+\infty$. Again by continuity, there exists $u=u_{0}>0$ such that $H(u_{0})=0$. Then $(u_{0}, v_{0})\in\mathbb{R}^{+}\times\mathbb{R}^{+}$ satisfies equations (\ref{5.2}).

In the following, we will prove that $u_{0}\neq1$, $v_{0}\neq1$ and $u_{0}\neq v_{0}$. In fact, if $u_{0}=1$, then $A_{1}=1/(2\sqrt{2})$, $A_{2}=1$, $A_{5}=A_{3}$. Hence,
$$
F(1,\,v)=\displaystyle8(m_{1}-m_{2})A_{1}+m_{2}-m_{1}=(m_{1}-m_{2})(2\sqrt{2}-1)=0,
$$
which contradicts $m_{1}\neq m_{2}$. Thus, $u_{0}\neq1$.

If $v_{0}=1$, then $A_{3}=1/(2\sqrt{2})$, $A_{4}=1$, $A_{5}=A_{1}$. Adding the two equations in (\ref{5.2}) gives
$$
F+G=\displaystyle(m_{2}-m_{5})(1-2\sqrt{2})=0,
$$
and it implies that $m_{2}=m_{5}$, contradicting the pairwise distinctness. Thus, $v_{0}\neq1$.

If $u_{0}=v_{0}$, then $A_{1}=A_{3}$, $A_{2}=A_{4}\neq0$, $A_{5}=A_{2}/(2\sqrt{2})$. The second equation in (\ref{5.2}) becomes
$$
G=(m_{1}-m_{5})(1-2\sqrt{2})A_{2}=0,
$$
which implies $m_{1}=m_{5}$, again a contradiction. Thus, $u_{0}\neq v_{0}$.

Take $a=\sqrt{u_{0}}$ and $a\neq1$. Choose $z_{5}$ such that $0<z_{5}^{2}=v_{0}$. Then $z_{5}=\sqrt{v_{0}}\neq1$ and $z_{5}=\sqrt{u_{0}}=\sqrt{v_{0}}=a$.
Finally, choose
$$
\lambda=\frac{2m_{1}}{(a^{2}+1)^{\frac{3}{2}}}+\frac{m_{2}}{4}+\frac{2m_{5}}{\left({z_5}^{2}+1\right)^{\frac{3}{2}}},
$$
then the first equation in (\ref{5.1}) holds, and since $F(u_{0},v_{0})=G(u_{0},v_{0})=0$, the second and third equations also equals $\lambda$. Therefore, for every triple of pairwise distinct positive masses $m_{1}$, $m_{2}$ and $m_{5}$,
there exist \( x_5 = y_5 = x_6 = y_6 = 0 \),  $z_6 = -z_5$ where $z_{5}\in\{z_{5}\in\mathbb{R}^{+}: z_{5}\neq1 \  {\rm and} \ z_{5} \neq a\}$, $0<a\neq1$ and $\lambda\in\mathbb{R}$ such that the six masses form a central configuration with a non-concyclic quadrilateral base. $\Box$

\section{Proof of Theorem \ref{thm2.4}}\label{sec6}\setcounter{equation}{0}
Recall that $a>0$, $x_5,\,x_6,\,y_5,\,y_6\in\mathbb{R}$, $z_{5}>0$ and $z_{6}\geq0$. To prove Theorem \ref{thm2.4}, we require the three following key propositions.

Recall that $ q_5 $ lies outside the plane of the rhombus. If $q_6 $ lies on the plane of the rhombus, we obtain the following proposition.
\begin{proposition} \label{prop6.1}
If $q_6 $ lies on the plane of the rhombus (see Figure \ref{fig6}), then the six masses cannot form a central configuration.
\end{proposition}

\textbf{Proof.}
We argue by contradiction. Since $q_6 $ lies on the plane of the rhombus, we have $z_6=0$.
By direct computation,
$$
f_{z_k}=\frac{m_5 z_5}{D_{k5}^3}-\lambda\left(0-\frac{m_5 z_5}{M}\right)=0,
\qquad k=1,2,3,4,6.
$$
Hence,
$$
D_{15}=D_{25}=D_{35}=D_{45}=D_{65},
$$
where
$$
\begin{cases}
D_{15}=\sqrt{{x_5}^{2}+(y_5+a)^{2}+{z_5}^{2}}, \quad
D_{25}=\sqrt{(x_5-1)^{2}+{y_5}^{2}+{z_5}^{2}}, \\[0.3cm]
D_{35}=\sqrt{{x_5}^{2}+(y_5-a)^{2}+{z_5}^{2}}, \quad
D_{45}=\sqrt{(x_5+1)^{2}+{y_5}^{2}+{z_5}^{2}}, \\[0.3cm]
D_{65}=\sqrt{(x_5-x_6)^{2}+(y_5-y_6)^{2}+{z_5}^{2}}.
\end{cases}
$$
From $D_{25}=D_{45}$, it follows that $x_5=0$.
Given that $a > 0$, and using the conditions $x_5 = 0$ and $D_{15} = D_{35}$, we conclude that $y_5 = 0$. Combined with $D_{15} = D_{25}$, this implies $a=1$.
Then, using $D_{15}=D_{65}$, we have
${x_6}^{2}+{y_6}^{2}=1.$
Noticed that
\begin{eqnarray*}
f_{y_k}= \displaystyle\sum\limits_{{j\neq k\atop{{1\leq j\leq 6}}}}\frac{m_{j} (y_{j}-y_{k})}{D_{kj}^{3}}+\lambda (y_{k}-y_{0})=0,
\qquad k=1,2,3,4,6,
\end{eqnarray*}
then using $x_5 = 0$ and $y_5 = 0$, we obtain
\begin{equation}\label{6.1}
\begin{cases}
\displaystyle f_{y_2}=\frac{a m_{3}}{D_{23}^{3}}-\frac{a m_{1}}{D_{21}^{3}}+\frac{y_6 m_{6}}{D_{26}^{3}}
-\frac{\lambda}{M}\bigl(a m_3-a m_1+y_6 m_6\bigr)=0, \\[0.5cm]
\displaystyle f_{y_4}=\frac{a m_{3}}{D_{43}^{3}}-\frac{a m_{1}}{D_{41}^{3}}+\frac{y_6 m_{6}}{D_{46}^{3}}
-\frac{\lambda}{M}\bigl(a m_3-a m_1+y_6 m_6\bigr)=0,
\end{cases}
\end{equation}
where
\[
D_{23}=D_{21}=D_{43}=D_{41}=\sqrt{2}, \qquad
D_{26}=\sqrt{(x_6-1)^2+{y_6}^{2}}, \qquad
D_{46}=\sqrt{(x_6+1)^2+{y_6}^{2}}.
\]
Subtracting the second equation from the first in equations \eqref{6.1}, we have
\[
y_6 \left\{\frac{1}{\bigl[(x_6-1)^2+{y_6}^{2}\bigr]^{3/2}}-\frac{1}{\bigl[(x_6+1)^2+{y_6}^{2}\bigr]^{3/2}}\right\} m_6=0.
\]
Since $m_6>0$, it follows that either $y_6=0$ or $x_6=0$. We now consider the two cases separately.  \\

\medskip
\textbf{Case 1.} Suppose $x_6=0$. Then, from ${x_6}^2+{y_6}^2=1$, we obtain $y_6=\pm 1$. Thus,
\[
q_6=(0,1,0) \quad \text{or} \quad q_6=(0,-1,0).
\]
Since $a=1$, the point $q_6$ coincides with either $q_3$ or $q_1$, contradicting the definition of a central configuration. \\

\medskip
\textbf{Case 2.} Suppose $y_6=0$. Then, from ${x_6}^2+{y_6}^2=1$, we obtain $x_6=\pm 1$. Thus,
\[
q_6=(1,0,0) \quad \text{or} \quad q_6=(-1,0,0).
\]
Again, since $a=1$, the point $q_6$ coincides with either $q_2$ or $q_4$, contradicting the definition of a central configuration.
\medskip
In both cases, we arrive at a contradiction. Therefore, our initial assumption $z_6=0$ must be false, completing the proof of Proposition~\ref{prop6.1}.
 $\Box$\\

To facilitate the later analysis, we now consider the case when $q_6$ lies off the plane of the rhombus and state two key propositions.
\begin{proposition} \label{prop6.2}
Assume $z_6>0$. No matter whether $ q_5 $ lies in the first, second, third or fourth octant, the six masses cannot form a central configuration.
\end{proposition}
\textbf{Proof.} Firstly, we employ a proof by contradiction to show that if $q_5$ lies in the first octant (i.e., $x_5 > 0$, $y_5 > 0$, and $z_5 > 0$; see Figure \ref{fig4}), then the six masses cannot form a central configuration.

Recall that
\begin{eqnarray}\label{6.2}
\left\{
  \begin{array}{lll}
f_{x_k}= \displaystyle\sum\limits_{{j\neq k\atop{{1\leq j\leq 6}}}}\frac{m_{j} (x_{j}-x_{k})}{D_{kj}^{3}}+\lambda (x_{k}-x_{0})=0,\\  [0.2cm]
f_{y_k}= \displaystyle\sum\limits_{{j\neq k\atop{{1\leq j\leq 6}}}}\frac{m_{j} (y_{j}-y_{k})}{D_{kj}^{3}}+\lambda (y_{k}-y_{0})=0,  \\[0.2cm]
f_{z_k}= \displaystyle\sum\limits_{{j\neq k\atop{{1\leq j\leq 6}}}}\frac{m_{j} (z_{j}-z_{k})}{D_{kj}^{3}}+\lambda (z_{k}-z_{0})=0,
  \end{array}
\right.
\end{eqnarray}
for $k = 1, 2, \dots, 6$. Here, $(x_0, y_0, z_0)$ is the center of mass, given by
 \begin{eqnarray*}
\displaystyle x_0= \frac{\sum_{1\leq j\leq 6}m_j x_j}{\sum_{1\leq j\leq 6} m_j}, \ \ y_0= \frac{\sum_{1\leq j\leq 6}m_j y_j}{\sum_{1\leq j\leq 6} m_j} \ \ \text{and} \ \ z_0= \frac{\sum_{1\leq j\leq 6}m_j z_j}{\sum_{1\leq j\leq 6} m_j}.
 \end{eqnarray*}
Then, the proof of Proposition \ref{prop6.2} proceeds in two steps.

\textbf{Step 1.} We show that under the assumption of Proposition \ref{prop6.2}, if the six masses form a central configuration, then $x_6<0$ and $y_6<0$.

By direction computation, we have
 \begin{eqnarray}\label{6.3}
 &&f_{z_2}-f_{z_4}= \displaystyle\left\{{\frac{1}{\left[\sqrt{(x_5-1)^2+{y_5}^2+{z_5}^{2}}\right]^3}-\frac{1}{\left[\sqrt{(x_5+1)^2+{y_5}^2+{z_5}^{2}}\right]^3}}\right\}m_5z_5 \nonumber \\[0.3cm]
 &&\quad\quad\quad+\left\{\frac{1}{\left[\sqrt{(x_6-1)^2+{y_6}^2+{z_6}^{2}}\right]^3}-\frac{1}{\left[\sqrt{(x_6+1)^2+{y_6}^2+{z_6}^{2}}\right]^3}\right\}m_6z_6=0
  \end{eqnarray}
and
 \begin{eqnarray}\label{6.4}
 &&f_{z_1}-f_{z_3}=\displaystyle\left\{\frac{1}{\left[\sqrt{{x_5}^2+(y_5-a)^2+{z_5}^{2}}\right]^3}-\frac{1}{\left[\sqrt{{x_5}^2+(y_5+a)^2+{z_5}^{2}}\right]^3}\right\}m_5z_5 \nonumber \\[0.3cm]
 &&\quad\quad\quad+\left\{\frac{1}{\left[\sqrt{{x_6}^2+(y_6-a)^2+{z_6}^{2}}\right]^3}-\frac{1}{\left[\sqrt{{x_6}^2+(y_6+a)^2+{z_6}^{2}}\right]^3}\right\}m_6z_6=0.
  \end{eqnarray}

On the one hand, by $x_5>0$, we have
\begin{eqnarray}\label{6.5}
\displaystyle\frac{1}{\left[\sqrt{(x_5-1)^2+{y_5}^2+{z_5}^{2}}\right]^3}-\frac{1}{\left[\sqrt{(x_5+1)^2+{y_5}^2+{z_5}^{2}}\right]^3}>0.
 \end{eqnarray}
Hence, combining $z_{5}>0$, $z_{6}>0$, $m_{5}>0$ and $m_{6}>0$, for equation (\ref{6.3}) to hold, the following inequality must be satisfied:
 \begin{eqnarray*}
\displaystyle\frac{1}{\left[\sqrt{(x_6-1)^2+{y_6}^2+{z_6}^{2}}\right]^3}-\frac{1}{\left[\sqrt{(x_6+1)^2+{y_6}^2+{z_6}^{2}}\right]^3}<0,
 \end{eqnarray*}
which implies that $x_6<0$.

On the other hand, using $y_5>0$, $a>0$ and equation (\ref{6.4}), one computes that
\begin{eqnarray}\label{6.6}
\displaystyle\frac{1}{\left[\sqrt{{x_5}^2+(y_5-a)^2+{z_5}^{2}}\right]^3}-\frac{1}{\left[\sqrt{{x_5}^2+(y_5+a)^2+{z_5}^{2}}\right]^3}>0.
 \end{eqnarray}
Under the conditions $z_5>0$, $z_{6}>0$, $a>0$, $m_5>0$, and $m_6>0$,
for equation (\ref{6.4}) to be valid, it follows that
 \begin{eqnarray*}
\displaystyle\frac{1}{\left[\sqrt{{x_6}^2+(y_6-a)^2+{z_6}^{2}}\right]^3}-\frac{1}{\left[\sqrt{{x_6}^2+(y_6+a)^2+{z_6}^{2}}\right]^3}<0.
 \end{eqnarray*}
Thus, $y_6<0$.

\textbf{Step 2.} We demonstrate that no central configuration exists under the assumptions of Proposition \ref{prop6.2}.

By direct computation, we have
\begin{eqnarray}\label{6.7}
 &&f_{y_2}-f_{y_4}=\left\{{\frac{1}{\Big[\sqrt{(x_5-1)^2+{y_5}^2+{z_5}^{2}}\Big]^3}-\frac{1}{\Big[\sqrt{(x_5+1)^2+{y_5}^2+{z_5}^{2}}\Big]^3}}\right\}m_5y_5 \nonumber \\[0.3cm]
 &&\quad\quad\quad+\left\{\frac{1}{\Big[\sqrt{(x_6-1)^2+{y_6}^2+{z_6}^{2}}\Big]^3}-\frac{1}{\Big[\sqrt{(x_6+1)^2+{y_6}^2+{z_6}^{2}}\Big]^3}\right\}m_6y_6=0.
  \end{eqnarray}
On the other hand, since $x_5>0$, $y_5>0$, $x_6<0$, $y_6<0$, $m_5>0$, and $m_6>0$, one can show that
\begin{eqnarray*}
 &&\displaystyle \left\{{\frac{1}{\Big[\sqrt{(x_5-1)^2+{y_5}^2+{z_5}^{2}}\Big]^3}-\frac{1}{\Big[\sqrt{(x_5+1)^2+{y_5}^2+{z_5}^{2}}\Big]^3}}\right\}m_5y_5 \nonumber \\[0.3cm]
 &&\quad\quad\quad+\left\{\frac{1}{\Big[\sqrt{(x_6-1)^2+{y_6}^2+{z_6}^{2}}\Big]^3}-\frac{1}{\Big[\sqrt{(x_6+1)^2+{y_6}^2+{z_6}^{2}}\Big]^3}\right\}m_6y_6>0,
\end{eqnarray*}
which contradicts equation (\ref{6.7}). So, if $q_5 $ lies in the first octant, then there is no central configuration.

%Through the above \textbf{Steps 1--2}, we have shown that if $q_5$ lies in the first octant, then six masses cannot form a central configuration. Subsequently, the same method is applied to the cases of the second, third, and fourth octants. The details are omitted due to the similarity in the procedure.
%By now, we conclude that if $q_5$ lies in any of the first four octants, the six masses cannot form a central configuration.

Through the above \textbf{Steps 1--2}, we have demonstrated that the six masses cannot form a central configuration if \( q_5 \) lies in the first octant.

The same method applies analogously to the cases where \( q_5 \) lies in the second, third, and fourth octants. The derivations are omitted due to the similarity in the reasoning.

We therefore conclude that no central configuration exists if \( q_5 \) is located in any of the first four octants. $\Box$

\begin{remark}\label{rem5.1}
The same methods can be applied to the cases where \( q_6 \) lies in the first, second, third, and fourth octants. We had same conclusions that no central configuration exists if \( q_6 \) is located in any of the first four octants.
\end{remark}

\begin{proposition} \label{prop6.3}
For a central configuration under equation (\ref{2.1}) with $z_{5}>0$ and $z_{6}\geq0$,
if $q_5$ lies on the positive $Z$-axis (see Figure \ref{fig5}),
then $q_6$ must also lie on the positive $Z$-axis, and the rhombus must be a square. Moreover, the values of the masses located at the four vertices of the square are equal.
\end{proposition}
\textbf{Proof.} Note that $q_5$ lies on the positive $Z$-axis and $q_6 $ lies off the plane containing the rhombus. Then, by equation (\ref{2.1}) and $z_{6}\geq0$, we get $x_5=y_5=0$, $z_{5}>0$ and $z_{6}>0$.
The remaining proof is divided into two steps.

\textbf{Step 1.} We prove that $q_6$ lies on the positive $Z$-axis and the rhombus is a square.

Observe that $x_5=0,$  $ z_5>0,$ and $ z_{6}>0$. Combining $m_6>0$, equation (\ref{6.3}) implies that
\begin{eqnarray}\label{6.8}
\left\{\frac{1}{\left[\sqrt{(x_6-1)^2+{y_6}^2+{z_6}^{2}}\right]^3}-\frac{1}{\left[\sqrt{(x_6+1)^2+{y_6}^2+{z_6}^{2}}\right]^3}\right\}m_6z_6=0.
\end{eqnarray}
It follows that $x_6=0$. Similarly, by $y_5=0$, equation (\ref{6.4}) can be rewritten as
\begin{eqnarray}\label{6.9}
\left\{\frac{1}{\left[\sqrt{{x_6}^2+(y_6-a)^2+{z_6}^{2}}\right]^3}-\frac{1}{\left[\sqrt{{x_6}^2+(y_6+a)^2+{z_6}^{2}}\right]^3}\right\}m_6z_6=0,
\end{eqnarray}
and it implies that $y_6=x_6=0$.
Hence, using $x_5=y_5=0$, the coordinates of $q_5$ and $q_6$ are $(0,0,z_{5})$ and $(0,0,z_{6})$. Therefore, $q_6$ lies on the positive $Z$-axis.

Next, we show that the rhombus is a square. In fact, %in (\ref{6.3}), subtracting the third equation from the second equation, we obtain
\begin{eqnarray}\label{6.10}
f_{z_2}-f_{z_3}=\displaystyle\Big[\frac{1}{\left(\sqrt{1+{z_5}^{2}}\right)^3}-\frac{1}{\left(\sqrt{a^2+{z_5}^{2}}\right)^3}\Big]m_5z_5+
\Big[\frac{1}{\left(\sqrt{1+{z_6}^{2}}\right)^3}-\frac{1}{\left(\sqrt{a^2+{z_6}^{2}}\right)^3}\Big]m_6z_6=0.
\end{eqnarray}
Note that $a>0$, $z_{5}>0, z_{6}>0$, $m_5>0$, and $m_6>0$, we have the following conclusions:
\begin{eqnarray*}
\left\{
\begin{array}{ll}
{\rm if} \,\,a>1,\,\,{\rm then}\,\,\displaystyle\frac{1}{\left(\sqrt{1+{z_5}^{2}}\right)^3}-\frac{1}{\left(\sqrt{a^2+{z_5}^{2}}\right)^3}>0 \ \ {\rm and} \ \ \frac{1}{\left(\sqrt{1+{z_6}^{2}}\right)^3}-\frac{1}{\left(\sqrt{a^2+{z_6}^{2}}\right)^3}>0,\nonumber \\[0.3cm]{\rm which}\,\,{\rm imply}\,\,{\rm that}\,\,{\rm equation}\,\, (\ref{6.10})\,\,{\rm does}\,\,{\rm not}\,\, {\rm hold}; \\[0.5cm]
{\rm if}\,\,0<a<1,\,\,{\rm then}\,\,\displaystyle\frac{1}{\left(\sqrt{1+{z_5}^{2}}\right)^3}-\frac{1}{\left(\sqrt{a^2+{z_5}^{2}}\right)^3}<0\,\,{\rm and}\,\,\frac{1}{\left(\sqrt{1+{z_6}^{2}}\right)^3}-\frac{1}{\left(\sqrt{a^2+{z_6}^{2}}\right)^3}<0,\nonumber \\[0.3cm]{\rm which}\,\,{\rm imply}\,\,{\rm that}\,\,{\rm equation} \,\, (\ref{6.10})\,\,{\rm does}\,\,{\rm not}\,\, {\rm hold}.
 \end{array}
\right.
\end{eqnarray*}
Hence, $a=1$. Then, using equation (\ref{2.1}), we conclude that if $q_5$ lies on the positive $Z$-axis, then the rhombus must be a square.

\textbf{Step 2.} We prove that $m_1=m_2=m_3=m_4$.

Firstly, we prove that $m_1 = m_3$ and $m_2 = m_4$.

From the \textbf{Step 1} in the proof of Proposition \ref{prop6.3}, we have $a=1$, and the coordinates of $q_5$ and $q_6$ are $(0,0,z_{5})$ and $(0,0,z_{6})$, respectively. Note that,
\begin{eqnarray*}
f_{x_5}-f_{x_6}= \Big[\frac{1}{\left(\sqrt{1+{z_5}^{2}}\right)^3}-\frac{1}{\left(\sqrt{1+{z_6}^{2}}\right)^3}\Big](m_2-m_4)=0,
\end{eqnarray*}
and
\begin{eqnarray*}
f_{y_5}-f_{y_6}= \Big[\frac{1}{\left(\sqrt{1+{z_5}^{2}}\right)^3}-\frac{1}{\left(\sqrt{1+{z_6}^{2}}\right)^3}\Big](m_3-m_1)=0.
\end{eqnarray*}
Since $z_{6}/z_{5}\neq 1$,  $1/\left(\sqrt{1+{z_5}^{2}}\right)^3-1/\left(\sqrt{1+{z_6}^{2}}\right)^3\neq0$. This implies that $m_2=m_4$ and $m_1=m_3.$

Secondly, we prove that $m_1=m_2=m_3=m_4$.

Using  $x_5=x_6=y_5=y_6=0$, $a=1$, $m_1 = m_3$ and $m_2 = m_4$, we have
\begin{eqnarray*}
f_{x_2} + f_{y_1}= (m_3-m_4)\left(\frac{1}{4}-\frac{\sqrt{2}}{2}\right)=0,
\end{eqnarray*}
which implies that $m_3=m_4. $ Hence, $m_1=m_2=m_3=m_4$. $\Box$\\

Now, we use Propositions \ref{prop6.1}--\ref{prop6.3} to prove Theorem \ref{thm2.4}.

\textbf{Proof of Theorem \ref{thm2.4}.} We recall that the rhombus is a square (i.e., $a=1$). Proposition \ref{prop6.2} and Remark \ref{rem5.1} show that if the six masses form a central configuration, then $q_5$ and $q_6$ must lie on coordinate planes. Moreover, Proposition \ref{prop6.1} establishes that for the central configuration, $q_6$ must be located off the $XOY$-plane, which implies that $z_{6}>0$. Consequently, using the assumption that $q_5$ is located off the $XOY$-plane, we have $q_5$ and $q_6$ are located either in the $YOZ$-plane or the $XOZ$-plane. Furthermore, for the central configuration, when $x_5=0$, it follows from equation (\ref{6.8}) that $x_6=0$. Similarly, if $ y_5 = 0$, it follows directly from equation (\ref{6.9}) that $y_6 = 0$.

Without loss of generality, we assume that $q_5$ and $q_6$ lie on the $YOZ$-plane, i.e.,
\begin{eqnarray*}
q_5 = (0, y_5, z_{5})\,\,{\rm and}\,\,q_6 = (0, y_6, z_{6}),
\end{eqnarray*}
where $z_{5}> 0$ and $z_{6} > 0$. In particular, if $y_5 = 0$, $q_5$ lies on $z$-axis. Proposition \ref{prop6.3} proves that $x_5 = x_6 = y_5=y_6=0$. Therefore, we only need to exclude the cases where $y_5 > 0$ and $y_5 < 0$.

Now we use a proof by contradiction to show that when \( y_5 > 0 \), the configuration $[q_1, q_2, \dots, q_6]$ is not a central configuration. Assuming that
$[q_1, q_2, \dots, q_6]$ is a central configuration with $x_5=x_6=0$ and $a=1$, we deduce a contradiction as follows.

From equations \eqref{6.2},
\begin{eqnarray}\label{6.11}
\left\{
\begin{array}{llll}
 f_{z_2}-f_{z_3}= \displaystyle\Big[\frac{1}{\left(\sqrt{1+{y_5}^2+{z_5}^{2}}\right)^{3}}-\frac{1}{\left(\sqrt{(y_5-1)^2+{z_5}^{2}}\right)^{3}}\Big]m_5z_5 +\\[0.4cm]
 \quad\quad\quad\quad\quad\quad \Big[\frac{1}{\left(\sqrt{1+{y_6}^2+{z_6}^{2}}\right)^{3}}-\frac{1}{\left(\sqrt{(y_6-1)^2+{z_6}^{2}}\right)^{3}}\Big]m_6z_6=0,\\[0.5cm]
f_{z_2}-f_{z_1}=  \displaystyle\Big[\frac{1}{\left(\sqrt{1+{y_5}^2+{z_5}^{2}}\right)^{3}}-\frac{1}{\Big(\sqrt{(y_5+1)^2+{z_5}^{2}}\Big)^{3}}\Big]m_5 z_5+\\[0.4cm]
 \quad\quad\quad\quad\quad\quad \Big[\frac{1}{\left(\sqrt{1+{y_6}^2+{z_6}^{2}}\right)^{3}}-\frac{1}{\Big(\sqrt{(y_6+1)^2+{z_6}^{2}}\Big)^{3}}\Big]m_6z_6=0.
 \end{array}
\right.
 \end{eqnarray}

 Denote
\begin{eqnarray*}
\left\{
  \begin{array}{lllll}
C=\displaystyle\frac{1}{\left[\sqrt{(y_5-1)^2+{z_5}^{2}}\right]^3}-\displaystyle\frac{1}{\left[\sqrt{1+{y_5}^2+{z_5}^{2}}\right]^3}, \\[0.3cm]
\widetilde{C}=\displaystyle\frac{1}{\left[\sqrt{1+{y_5}^2+{z_5}^{2}}\right]^3}-\frac{1}{\left[\sqrt{(y_5+1)^2+{z_5}^{2}}\right]^3}, \\[0.3cm]
D=\displaystyle\frac{1}{\left[\sqrt{(y_6+1)^2+{z_6}^{2}}\right]^3}-\frac{1}{\left[\sqrt{1+{y_6}^2+{z_6}^{2}}\right]^3}, \\[0.3cm]
\widetilde{D}=\displaystyle\frac{1}{\left[\sqrt{1+{y_6}^2+{z_6}^{2}}\right]^3}-\frac{1}{\left[\sqrt{(y_6-1)^2+{z_6}^{2}}\right]^3}.
\end{array}
\right.
\end{eqnarray*}
Let matrix $B$ be defined as follows:
\begin{eqnarray}\label{6.12}
B=\left(
     \begin{array}{cc}
   -C &  \widetilde{C} \\
  \widetilde{D} &  -D
 \end{array}
   \right).
\end{eqnarray}
Then, from system (\ref{6.11}) and equation (\ref{6.12}) it follows that
\begin{eqnarray*}
B\left(
     \begin{array}{c}
  m_5z_5 \\
 m_6 z_6
 \end{array}
   \right)=\left(
     \begin{array}{c}
 0 \\
 0
 \end{array}
   \right)
\end{eqnarray*}
and the determinant of matrix $B$ is $|B|= CD-\widetilde{C}\widetilde{D}$.

\begin{eqnarray*}
C-\widetilde{C}=\frac{1}{\left[\sqrt{(y_5-1)^2+{z_5}^{2}}\right]^3}-\displaystyle\frac{2}{\left[\sqrt{1+{y_5}^2+{z_5}^{2}}\right]^3}
+\frac{1}{\left[\sqrt{(y_5+1)^2+{z_5}^{2}}\right]^3}.
\end{eqnarray*}
Let $1+{y_5}^2+{z_5}^{2}=\xi>0$. Then $\left(\sqrt{1+2y_5/\xi}\right)^{3}\left(\sqrt{\xi-2{y_5}}\right)^3>0$, $0<y_5/\xi<1/2$ and

\begin{eqnarray*}
C-\widetilde{C}=\frac{\left(\sqrt{1+2\frac{y_5}{\xi}}\right)^3-2\left(\sqrt{1-2\frac{y_5}{\xi}}\right)^3\left(\sqrt{1+2\frac{y_5}{\xi}}\right)^3+\left(\sqrt{1-2\frac{y_5}{\xi}}\right)^3}
{\left(\sqrt{1+2\frac{y_5}{\xi}}\right)^3\left(\sqrt{\xi-2{y_5}}\right)^3}.
\end{eqnarray*}
Moreover, denote $v=1+2y_5/\xi$, then $1-2y_5/\xi=2-v$ and
\begin{eqnarray*}
h(v)=\left(\sqrt{v}\right)^3-2\left(\sqrt{2-v}\right)^{3}\left(\sqrt{v}\right)^{3}+\left(\sqrt{2-v}\right)^{3},
\end{eqnarray*}
where $1<v<2$ and $ 0<2-v<1$. So, the derivative of $h(v)$ is given by
\begin{eqnarray*}
h^{\prime}(v)=\frac{3}{2}\sqrt{v}-\frac{3}{2}\sqrt{2-v}-3\sqrt{v}\sqrt{2-v}(2-2v).
 \end{eqnarray*}
Since $1 < v < 2$, we have
\begin{eqnarray*}
0 < 2 - v < 1 ,\,\, -2 < 2 - 2v < 0.
\end{eqnarray*}
This implies that
\begin{eqnarray*}
\frac{3}{2}\sqrt{v} - \frac{3}{2}\sqrt{2 - v} > 0,\,\,-3\sqrt{v}\sqrt{2 - v} (2 - 2v) > 0.
\end{eqnarray*}
Consequently, $h^{\prime}(v) > 0$, which implies that the function $h(v)$ increases on the interval $(1, 2)$.

Furthermore, as $v \rightarrow 1 $, $h(v) \rightarrow0 $. Since $h(v) $ increases monotonically, it follows that for all $ v \in (1, 2)$,
$h(v) > 0$, i.e., $C - \widetilde{C} > 0$.

Similarly, we can prove that $D-\widetilde{D}>0$.  Then, since $C> \widetilde{C}>0$, and $D> \widetilde{D}>0$, it follows that  $|B|= CD-\widetilde{C}\widetilde{D}>0$. Thus, system (\ref{6.11}) has only zero solution $(m_5z_5,m_6z_6)^{T}=(0,0)^{T}$, which contradicts $m_5>0$ and $m_6>0$. Hence, when \( y_5 > 0 \), $[q_1, q_2, \dots, q_6]$ is not a central configuration.

By similar argument, if $y_5 < 0$, we can deduce the non-existence of the central configuration. This completes the proof of Theorem \ref{thm2.4}. $\Box$\\

\section{Proof of Theorem \ref{thm2.5}}\label{sec7}\setcounter{equation}{0}

\textbf{Proof.} We use Propositions \ref{prop6.1}--\ref{prop6.3} to prove Theorem \ref{thm2.5}.

{\bf To prove Theorem \ref{thm2.5} (1):} From Proposition \ref{prop6.2}, we see that $q_5$ must lie on a coordinate plane. Then, combining $q_5= (x_5, y_5,z_{5})$ and $z_{5}>0$, it follows that $q_5$ lies in either the $XOZ$-plane or the $YOZ$-plane.
Recall that $z_{6}>0$. Moreover, if $q_5$ belongs to the $XOZ$-plane, then $y_5=0$. Then, from equation (\ref{6.4}) and $z_{5},z_{6}\neq0$, we have $y_6=0$. Thus, $q_6=(x_6,0,z_{6})$, i.e., $q_6$ also belongs to the $XOZ$-plane. On the other hand, if $q_5$ belongs to the $YOZ$-plane, then $x_5=0$. From equation (\ref{6.3}) and $z_{5},z_{6}\neq0$, we have $x_6=0$. Therefore, $q_6=(0,y_6,z_{6})$, i.e., $q_6$ also belongs to the $YOZ$-plane.

By now, we conclude that $q_5$ and $q_6$ both lie in $XOZ$-plane or both lie in the $YOZ$-plane.

In what follows, for the spatial central configuration, no matter if $q_{5}$ lies on the positive $Z$-axis, then the rhombus must be a square. Then, order to prove Theorem \ref{thm2.5} (2) and (3), it suffices to consider the case where $q_5$ is located off the plane containing the rhombus.

{\bf To prove Theorem \ref{thm2.5} (2):}
 When $q_5$ belongs to the $XOZ$-plane, then by Theorem \ref{thm2.5} (1), we have $q_6=(x_6,0,z_{6})$. The remaining proof of Theorem \ref{thm2.5} (2) is discussed in three cases:

\medskip
\textbf{Case 1.} Assume $x_5 = 0$, since $q_5$ is located off the plane containing the rhombus with $y_5 = 0$, and $z_{5},z_{6}>0$, Proposition \ref{prop6.3} ({\bf Step 1}) would force the rhombus to be a square, which contradicts the assumption of Theorem \ref{thm2.5}. Therefore, $x_5\neq0$.

\medskip
\textbf{Case 2.} Assume $x_5 > 0$. Then inequality (\ref{6.5}) holds, and from the first step in the proof of Proposition \ref{prop6.2}, we have $x_6 < 0$.

\medskip
\textbf{Case 3.} Assume $x_5<0$, then
\begin{eqnarray*}
\displaystyle\frac{1}{\left[\sqrt{(x_5-1)^2+{y_5}^2+{z_5}^{2}}\right]^3}-\frac{1}{\left[\sqrt{(x_5+1)^2+{y_5}^2+{z_5}^{2}}\right]^3}<0.
 \end{eqnarray*}
Hence, combining $z_{6}/z_{5}>0$, $m_{5}>0$ and $m_{6}>0$, for equation (\ref{6.3}) to hold, we must have
 \begin{eqnarray*}
\displaystyle\frac{1}{\left[\sqrt{(x_6-1)^2+{y_6}^2+{z_6}^{2}}\right]^3}-\frac{1}{\left[\sqrt{(x_6+1)^2+{y_6}^2+{z_6}^{2}}\right]^3}>0,
 \end{eqnarray*}
which implies that $x_6>0$.

From the \textbf{Cases 1-3}, we have $x_5x_6<0$.

In the following, for the case of \( q_5 \) and \( q_6 \) both lie in the $XOZ$-plane, we use $x_5x_6<0$ and $y_{5}=y_{6}=0$ to prove that $m_{1}=m_{3}$, $m_{2}=m_{4}$ and $m_{5}=m_{6}$.

Firstly, we show that $m_{1}=m_{3}$ by contradiction, and we assume that $m_{1}\neq m_{3}$.

From equation (\ref{2.1}), $y_{5}=y_{6}=0$ and $x_{0}=(\sum_{1\leq j\leq 6}m_{j}x_{j})/(\sum_{1\leq j\leq 6}m_j)$, then $f_{x_{3}}=0$ implies that
\begin{eqnarray*}
\bigg[\frac{1}{(1+a^2)^{\frac{3}{2}}}-\frac{\lambda}{\displaystyle\sum_{1\leq j\leq 6}m_j}\bigg](m_1-m_3)=0.
\end{eqnarray*}
Therefore, $m_1=m_3$ or $\lambda=[\sum_{1\leq j\leq 6}m_j]/\left(1+a^2\right)^{3/2}$. Thanks to
$m_1\neq m_3$, we have $\lambda=[\sum_{1\leq j\leq 6}m_j]/\left(1+a^2\right)^{3/2}$.

Recall that $y_{5}=y_{6}=0$, then by direct computation, we have
\begin{equation*}
\left\{
  \begin{array}{ll}
D_{53}^2-D_{51}^2=\left[x_5^2+\left(y_5-a\right)^2+z_5{ }^2\right]-\left[x_5^2+\left(y_5-a\right)^2+z_5{ }^2\right]=-4 a y_5=0,\\
D_{63}^2-D_{61}^2=\left[x_6^2+\left(y_6-a\right)^2+z_6{ }^2\right]-\left[x_6^2+\left(y_6-a\right)^2+z_6{ }^2\right]=-4 a y_6=0,
  \end{array}
\right.
\end{equation*}
and it implies that $1/D^3_{53}=1/D^3_{51}$ and $1/D^3_{63}=1/D^3_{61}$. Hence with the aid of $1/D_{12}^3=1/D_{14}^3=1/D_{32}^3=1/D_{34}^3=(1+a^2)^{3/2}$ and $f_{y_{1}}=f_{y_{3}}=0$, one computes that
\begin{eqnarray*}
(m_3-m_1)(\frac{1}{D_{31}^3}-\frac{\lambda}{\displaystyle\sum_{1\leq j\leq 6}m_j})=0.
\end{eqnarray*}
Note that $m_{1}\neq m_{3}$, then thanks to equation (\ref{2.1}) and $\lambda=[\sum_{1\leq j\leq 6}m_j]/\left(1+a^2\right)^{3/2}$, one can verify that $\sqrt{1+a^{2}}=2a$, which implies that $a=\sqrt{3}/3$ and $\lambda=[3\sqrt{3}\sum_{1\leq j\leq 6}m_j]/8$.

On the other hand, note that equations
(\ref{1.2}), $D_{15}=D_{51}$, $D_{35}=D_{53}$, $D_{16}=D_{61}$ and $D_{56}=D_{65}$, then from $f_{y_{5}=0}$ and $f_{y_{6}=0}$, we have
\begin{equation}\label{7.1}
\left\{
  \begin{array}{ll}
\displaystyle\frac{-am_1}{D_{51}^3}+\frac{am_3}{D_{53}^3}=\frac{\lambda}{\displaystyle\sum_{1\leq j\leq 6}m_j}\left(am_3-am_1\right),\\
\displaystyle\frac{-am_1}{D_{61}^3}+\frac{am_3}{D_{63}^3}=\frac{\lambda}{\displaystyle\sum_{1\leq j\leq 6}m_j}\left(am_3-am_1\right).
  \end{array}
\right.
\end{equation}
Inserting $1/D_{51}^3=1/D_{53}^3$, $1/D_{61}^3=1/D_{63}^3$ and $a>0$ into equation (\ref{7.1}), one computes that
\begin{equation*}
\left\{
  \begin{array}{ll}
\displaystyle(m_3-m_1)(\frac{1}{D_{51}^3}-\frac{\lambda}{\displaystyle\sum_{1\leq j\leq 6}m_j})=(m_3-m_1)(\frac{1}{D_{53}^3}-\frac{\lambda}{\displaystyle\sum_{1\leq j\leq 6}m_j})=0,\\
\displaystyle(m_3-m_1)(\frac{1}{D_{61}^3}-\frac{\lambda}{\displaystyle\sum_{1\leq j\leq 6}m_j})=(m_3-m_1)(\frac{1}{D_{63}^3}-\frac{\lambda}{\displaystyle\sum_{1\leq j\leq 6}m_j})=0.
  \end{array}
\right.
\end{equation*}
Then with the help of $m_{1}\neq m_{3}$ and $\lambda=3\sqrt{3}[\sum_{1\leq j\leq 6}m_j]/8$, we have $D_{51}=D_{53}=D_{61}=D_{63}=2/\sqrt{3}$.

Note that in this situation, $y_{5}=0$ and $m_{1}\neq m_{3}$, thus $a^{2}=1/3$ and $D_{51}=D_{53}=2/\sqrt{3}$. Then the triangle $q_{5}q_{1}q_{3}$ is an isosceles triangle (see Figure \ref{fig4}). Denote the midpoint of segment
$q_{1}q_{3}$ by point $O$, and then $O=(0,0,0)$. Thus, $|q_{1}-O|=D_{13}/2=\sqrt{3}/3$ and $q_{5}O$ is perpendicular to the plane which containing the rhombus $q_{1}q_{2}q_{3}q_{4}$, and then
 $|q_{5}-O|=\sqrt{D^{2}_{51}-|q_{1}-O|^{2}}=1$. Moreover, in the right triangle $q_{5}q_{1}O$, we have $D^{2}_{51}=|q_{1}-O|^{2}+|q_{5}-O|^{2}=(\sqrt{3}/3)^{2}+1=4/3$. On the other hand, by equation (\ref{2.1}) and $y_{5}=0$, it is easy to see that $D^{2}_{51}=x^{2}_{5}+a^{2}+z^{2}_{5}=x^{2}_{5}+z^{2}_{5}+1/3$, which implies that $x^{2}_{5}+z^{2}_{5}=1$.

In the right triangle $q_{5}q_{2}O$, by $q_{2}=(1,0,0)$, $O=(0,0,0)$ and $|q_{5}-O|=1$, then $D^{2}_{52}=|q_{5}-q_{2}|^{2}=|q_{2}-O|^{2}+|q_{5}-O|^{2}=1+1=2$. Combining $y_{5}=0$, equation (\ref{2.1}) and $x^{2}_{5}+z^{2}_{5}=1$, we have
\begin{eqnarray*}
2=D^{2}_{52}=(x_{5}-1)^{2}+z^{2}_{5}=(x^{2}_{5}+z^{2}_{5})-2x_{5}+1=2-2x_{5},
\end{eqnarray*}
which implies that $x_{5}=0$, and it contradicts with the assumption that $x_{5}\neq0$.
So $m_{1}= m_{3}$.

Secondly, based upon $m_{1}=m_{3}$, we deduce $m_{2}=m_{4}$ as follows.

By equation (\ref{2.1}) and $y_{5}=y_{6}=0$, we have $D_{51}=D_{53}$ and $D_{61}=D_{63}$. Then, inserting $m_{1}= m_{3}$ and $y_{5}=y_{6}=0$ into $f_{y_{5}}+f_{y_{6}}=0$, we have
$$
\frac{-am_{1}}{D^{3}_{51}}+\frac{-am_{3}}{D^{3}_{53}}+\frac{am_{1}}{D^{3}_{61}}+\frac{am_{3}}{D^{3}_{63}}=\frac{-2am_{1}}{D^{3}_{51}}+\frac{am_{1}}{D^{3}_{61}}=0,
$$
which implies that $D_{51}=D_{61}$.

 Using $f_{x_{1}}=0$, $f_{y_{1}}=0$  and $f_{z_{1}}=0$, one computes that
\begin{equation}\label{7.2}
\left\{
  \begin{array}{lll}
\displaystyle\frac{m_{2}}{D^{3}_{21}}+\frac{-m_{4}}{D^{3}_{41}}+\frac{m_{5}x_{5}}{D^{3}_{51}}+\frac{m_{6}x_{6}}{D^{3}_{61}}=\lambda\displaystyle\frac{m_{2}-m_{4}+m_{5}x_{5}+m_{6}x_{6}}{M},\\[0.3cm]
\displaystyle\frac{m_{2}a}{D^{3}_{21}}+\frac{2m_{3}a}{D^{3}_{31}}+\frac{m_{4}a}{D^{3}_{41}}+\frac{m_{5}a}{D^{3}_{51}}+\frac{m_{6}a}{D^{3}_{61}}=\lambda a,\\[0.3cm]
\displaystyle \frac{m_{5}z_{5}}{D^{3}_{51}}+\frac{m_{6}z_{6}}{D^{3}_{61}}=\lambda\displaystyle\frac{m_{5}z_{5}+m_{6}z_{6}}{M}.
  \end{array}
\right.
\end{equation}
Recall that $z_{5}>0$, $z_{6}>0$ and $D_{51}=D_{61}$. Then the third equation of (\ref{7.2}) implies that $\lambda/M=1/D^{3}_{51}=1/D^{3}_{61}$. Thus, using $D_{21}=D_{41}$ and $m_{1}=m_{3}$, the first and second equations are reduced to
\begin{equation}\label{7.3}
\left\{
  \begin{array}{lll}
\displaystyle\frac{m_{2}-m_{4}}{D^{3}_{21}}=\lambda\displaystyle\frac{m_{2}-m_{4}}{M}=\displaystyle\frac{m_{2}-m_{4}}{D^{3}_{51}},\\[0.3cm]
\displaystyle\frac{m_{2}+m_{4}}{D^{3}_{21}}+\frac{2m_{3}}{D^{3}_{31}}=\frac{m_{2}+2m_{3}+m_{4}}{D^{3}_{51}}.
  \end{array}
\right.
\end{equation}

In the following, we prove $m_{2}=m_{4}$ by contradiction, and we assume that $m_{2}\neq m_{4}$.

Note that $m_{2}\neq m_{4}$, then from the first equation of (\ref{7.3}), we obtain $D_{21}=D_{51}$. Combining the second equation of (\ref{7.3}), we have $D_{34}=D_{21}=D_{31}$. Then using equation (\ref{2.1}), we see that the triangle $q_{1}q_{3}q_{4}$ is an equilateral triangle (see Figure \ref{fig4}). This implies that in the right triangle $q_{1}q_{4}O$, the angle $\angle q_{1}q_{4}O=\pi/6$, and then $|Oq_{1}|=|Oq_{4}|$. So $a=1$, and it contradicts $a\neq1$. Hence, $m_{2}=m_{4}$.

Thirdly, we use $y_{5}=y_{6}=0$, $x_{5}x_{6}<0$, $m_{1}=m_{3}$ and $m_{2}=m_{4}$ to prove $\sqrt{3}/3<a<1$, $m_{5}=m_{6}$, $x_{5}=-x_{6}$ and $z_{5}=z_{6}$ as the following.

Recall that the masses satisfy $m_{1}=m_{3}$ and $m_{2}=m_{4}$. Since $x_{5}x_{6}<0$, $y_{5}=y_{6}=0$, $z_{5}>0$, and $z_{6}>0$, without loss of generality, we assume the six positions are
\begin{eqnarray*}
\left\{
  \begin{array}{ll}
q_1=(0,-a,0), \,\, q_2=(1,0,0), \,\, q_3=(0,a,0),\\ [0.3cm]
q_4=(-1,0,0), \,\, q_5=(x_5,0,z_5), \,\, q_6=(x_6,0,z_6),
  \end{array}
\right.
\end{eqnarray*}
where $x_5>0$, $x_6<0$, $z_5>0$, and $z_6>0$.

For the central configuration, by \eqref{1.2}, there exists a constant $\lambda\in\mathbb{R}$ such that
\begin{eqnarray}\label{7.4}
\sum\limits_{{j\neq k\atop{{1\leq j\leq 6}}}}
\frac{m_{j} (q_{j}-q_{k})}{|q_{j}-q_{k}|^{3}}+\lambda (q_{k}-q_{0})=0, \quad \  \  k=1,2,\cdots,6,
\end{eqnarray}
where the barycenter $q_{0}$ is given by $q_{0}=(\sum_{1\leq j\leq 6}m_{j}q_{j})/M$ with $M=\sum_{1\leq j\leq 6}m_{j}$. Then
using
$M(q_k-q_0)=\sum_{1\leq j\neq k\leq 6}m_j(q_k-q_j)$, we obtain
\begin{equation}\label{7.5}
\sum\limits_{{j\neq k\atop{{1\leq j\leq 6}}}}m_j(q_j-q_k)(D_{kj}^{-3}-\frac{\lambda}{M})=0,
\qquad k=1,2,\cdots,6.
\end{equation}

For $k=1$, the $x$ and $z$ components of \eqref{7.5} are
\[
m_{5}x_{5}(\frac{1}{D^{3}_{15}}-\frac{\lambda}{M})+m_{6}x_{6}(\frac{1}{D^{3}_{16}}-\frac{\lambda}{M})=0, \quad
m_{5}z_{5}(\frac{1}{D^{3}_{15}}-\frac{\lambda}{M})+m_{6}z_{6}(\frac{1}{D^{3}_{16}}-\frac{\lambda}{M})=0.
\]
The determinant of this system is $m_{5}m_{6}(x_{5}z_{6}-x_{6}z_{5})>0$.
Consequently, $D_{15}^{-3}=D_{16}^{-3}=\lambda/M>0$, and hence
\begin{equation}\label{7.6}
{x_5}^2+{z_5}^2={x_6}^2+{z_6}^2=:R^2,
\qquad \frac{\lambda}{M}=(a^2+R^2)^{-3/2}.
\end{equation}
In particular, $D_{35}^{-3}=D_{36}^{-3}=\lambda/M$ as well.

Let $c=1+R^2$ and we define
\[
\Theta(t)=(c-2t)^{-3/2}+(c+2t)^{-3/2},\ \
J(t)=(c-2t)^{-3/2}-(c+2t)^{-3/2}\quad \text{and} \quad
K(t)=\Theta(t)-2\frac{\lambda}{M},
\]
where $t>0$.
These functions will only be evaluated at $t=x_{5}$ and $-x_{6}$.
Their denominators are positive because
$c-2x_{5}=D_{25}^2>0$ and $c+2x_{6}=D_{46}^2>0$.
Since $x_{5}>0,-x_{6}>0$, we have $J(x_{5}),J(-x_{6})>0$.
The remaining distances to the horizontal pair satisfy
\[
D_{25}^2=c-2x_{5}, \ \  D_{45}^2=c+2x_{5}, \ \
D_{46}^2=c+2x_{6}, \ \ D_{26}^2=c-2x_{6}.
\]

The $z$ components of \eqref{7.5} for $k=2,4$ are
\begin{equation*}
\left\{
  \begin{array}{ll}
\displaystyle m_{5}z_{5}(\frac{1}{D^{3}_{25}}-\frac{\lambda}{M})+m_{6}z_{6}(\frac{1}{D^{3}_{26}}-\frac{\lambda}{M})=0,\\[0.3cm]
\displaystyle m_{5}z_{5}(\frac{1}{D^{3}_{45}}-\frac{\lambda}{M})+m_{6}z_{6}(\frac{1}{D^{3}_{46}}-\frac{\lambda}{M})=0.
  \end{array}
\right.
\end{equation*}
Adding and subtracting these equations, and adding the $x$
components for $k=2,4$, respectively, yields
\begin{equation}\label{7.7}
\left\{
  \begin{array}{lll}
m_{5}z_{5}K(x_{5})+m_{6}z_{6}K(-x_{6})=0,\\[0.3cm]
m_{5}z_{5}J(x_{5})=m_{6}z_{6}J(-x_{6}),\\[0.3cm]
m_{5}[x_{5}K(x_5)-J(x_{5})]-m_{6}[-x_{6}K(-x_{6})-J(-x_{6})]=0.
  \end{array}
\right.
\end{equation}
Set $P=x_{5}z_{6}-x_{6}z_{5}>0$ and $\eta=(z_{6}-z_{5})/P$.
The first two equations in \eqref{7.7} give
$K(x_{5})/J(x_{5})=-K(-x_{6})/J(-x_{6})$.
Writing this common ratio as $s$ and substituting it into the
third equation, we obtain
\[
0=m_{5}J(x_{5})(x_{5}s-1)+m_{6}J(-x_{6})(-x_{6}s+1)
=\frac{m_{6}J(-x_{6})}{z_{5}}\bigl[Ps+z_{5}-z_{6}\bigr].
\]
It follows that $s=\eta$, and therefore
\begin{equation}\label{7.8}
K(x_{5})=\eta J(x_{5}),\qquad K(-x_{6})=-\eta J(-x_{6}).
\end{equation}

On the one hand, the first equation in \eqref{7.7} implies
\[
\frac{\lambda}{M}=\frac{m_{5}z_{5}\Theta(x_{5})+m_{6}z_{6}\Theta(-x_{6})}{2(m_{5}z_{5}+m_{6}z_{6})}.
\]
Strict convexity of $t\mapsto t^{-3/2}$ on $(0,\infty)$ gives
$\Theta(x_{5})>\Theta(0)=2c^{-3/2}$ and $\Theta(-x_{6})>\Theta(0)=2c^{-3/2}$.
Thus $\lambda/M>c^{-3/2}$. Together with \eqref{7.6},
we have $a^2<1$. On the other hand,
using $a^{2}<1$ and $c=1+R^{2}$, we have $D_{45}^2=c+2x_{5}>a^2+R^2$,
$D_{26}^2=c-2x_{6}>a^2+R^2$, $D_{45}^{-3}-\lambda/M<0$ and $D_{26}^{-3}-\lambda/M<0$.
By \eqref{7.8}, one computes that
\begin{equation*}
\left\{
  \begin{array}{ll}
\displaystyle 2\frac{\lambda}{M}=\Theta(x_{5})-\eta J(x_{5})=\frac{1}{D^{3}_{52}}+\frac{1}{D^{3}_{54}}-\eta(\frac{1}{D^{3}_{52}}-\frac{1}{D^{3}_{54}}),\\[0.3cm]
\displaystyle 2\frac{\lambda}{M}=\Theta(-x_{6})-\eta J(-x_{6})=\frac{1}{D^{3}_{64}}+\frac{1}{D^{3}_{62}}-\eta(\frac{1}{D^{3}_{64}}-\frac{1}{D^{3}_{62}}),
  \end{array}
\right.
\end{equation*}
which implies
\[ \frac{1}{D^{3}_{45}}-\frac{\lambda}{M}=\frac{\eta-1}{2}J(x_{5}),\qquad
\frac{1}{D^{3}_{26}}-\frac{\lambda}{M}=-\frac{\eta+1}{2}J(-x_{6}).
\]
Consequently, $-1<\eta<1$.

Recall that $D_{15}^{-3}=D_{16}^{-3}=\lambda/M$. The $y$ component of \eqref{7.5} for $k=1$ gives
\[
m_{1}\Big((2a)^{-3}-\frac{\lambda}{M}\Big)+m_{2}\Big((1+a^2)^{-3/2}-\frac{\lambda}{M}\Big)=0.
\]
Note that $P=x_{5}z_{6}-x_{6}z_{5}$, then for $k=2$, its $x$ component is
\[
\begin{aligned}
2m_{1}\Big((1+a^2)^{-3/2}-\frac{\lambda}{M}\Big)+2m_{2}(\frac{1}{8}-\frac{\lambda}{M})
&=m_{5}(x_{5}-1)(\frac{1}{D^{3}_{25}}-\frac{\lambda}{M})
  -m_{6}(-x_{6}+1)(\frac{1}{D^{3}_{26}}-\frac{\lambda}{M})\\
&=\frac{1+\eta}{2}
  \bigl[m_{5}(x_{5}-1)J(x_{5})+m_{6}(-x_{6}+1)J(-x_{6})\bigr]\\
&=\frac{m_{6}J(-x_{6})(1+\eta)}{2z_{5}}
  \bigl[z_{6}(x_{5}-1)+z_{5}(-x_{6}+1)\bigr]\\
&=\frac{m_{6}PJ(-x_{6})}{2z_5}(1-\eta^2),
\end{aligned}
\]
where we used $m_{5}z_{5}J(x_{5})=m_{6}z_{6}J(-x_{6})$ and $z_{6}-z_{5}=\eta P$.
We have therefore proved
\begin{equation}\label{7.9}
\left\{
  \begin{array}{ll}
\displaystyle m_{1}\Big((2a)^{-3}-\frac{\lambda}{M}\Big)+m_{2}\Big((1+a^2)^{-3/2}-\frac{\lambda}{M}\Big)=0,\\[0.3cm]
\displaystyle m_{1}\Big((1+a^2)^{-3/2}-\frac{\lambda}{M}\Big)+m_{2}(\frac{1}{8}-\frac{\lambda}{M})
=\frac{m_{5}PJ(-x_{6})}{4z_{5}}(1-\eta^2)>0.
  \end{array}
\right.
\end{equation}

Suppose that $R^2\le1$.
Then \eqref{7.6} gives $\lambda/M\ge (1+a^2)^{-3/2}$, and the first equation
in \eqref{7.9}, with $m_{1},m_{2}>0$, implies $(2a)^{-3}\ge\lambda/M$.
Thus $3a^2\le R^2\le1$, so that $a^2+R^2\le4/3$, and then
$\lambda/M\ge(4/3)^{-3/2}>1/8$.
The left-hand side of the second equation in \eqref{7.9}
would then be strictly negative, a contradiction.
Hence $R^2>1$.
It follows that $\lambda/M<(1+a^2)^{-3/2}$, so the first equation in
\eqref{7.9} gives $(2a)^{-3}<\lambda/M$.
By \eqref{7.6}, this means $R^2<3a^2$.
In particular, $a>\sqrt{3}/3$, and the necessary bound
$\sqrt{3}/3<a<1$ follows.

Recall that $D_{56}^2=(x_{5}-x_{6})^2+(z_{5}-z_{6})^2$.
Because $D_{15}^{-3}=D_{35}^{-3}=\lambda/M$, the $x$ component of
\eqref{7.5} for $k=5$ reduces to
\begin{equation}\label{7.10}
\displaystyle m_{2}[J(x_{5})-x_{5}K(x_{5})]=m_{6}(x_{5}-x_{6})(D^{-3}_{56}-\frac{\lambda}{M}).
\end{equation}
By $P=x_{5}z_{6}-x_{6}z_{5}>0$ and $\eta=(z_{6}-z_{5})/P$, we have $1-x_{5}\eta=z_{5}(x_{5}-x_{6})/P$.
Inserting \eqref{7.8} into \eqref{7.10} , we obtain
$D^{-3}_{56}-\lambda/M=m_{2}z_{5}J(x_{5})/(m_{6}P)>0$. Using $\lambda/M=(a^2+R^2)^{-3/2}$ and $c=1+R^{2}$,
we have thus established
\begin{equation}\label{7.11}
R^2>1,\qquad D^2_{56}<a^2+R^2<c,\qquad -1<\eta<1.
\end{equation}

It remains to exclude $x_{5}\ne -x_{6}$. Since the treatment of $x_{6}>-x_{5}$ is analogous to that of $x_{5}>-x_{6}$, here we only consider the case of $x_{5}>-x_{6}$. Note that $-x_6, \, x_5>0$,
we may assume that $x_{5}>-x_{6}>0$.
Then \eqref{7.6} gives $0<z_{5}<z_{6}$, and hence $0<\eta<1$.
For $\beta>0$, denote
\[
G_\beta=F_\beta(x_{5})-F_\beta(-x_{6})
-\eta\bigl[J_\beta(x_{5})+J_\beta(-x_{6})\bigr],
\]
where
\[
F_\beta(t):=(c-2t)^{-\beta}+(c+2t)^{-\beta},  \quad
J_\beta(t):=(c-2t)^{-\beta}-(c+2t)^{-\beta}, \quad t>0.
\]
Subtracting the two equations in \eqref{7.8} gives
$G_{3/2}=0$.
We will show that $G_{3/2}>0$ instead.

For the auxiliary exponent $\beta=1$, we have
\[
F_1(x)=\frac{2c}{c^2-4x^2},
\qquad J_1(x)=\frac{4x}{c^2-4x^2}.
\]
Using $0<\eta<1$ and $\eta=(z_{6}-z_{5})/P$, then $z_{6}-z_{5}>0$. Combining ${z_6}^2-{z_5}^2={x_5}^2-{x_6}^2$, $(x_{5}+x_{6})/(z_{6}-z_{5})=(z_{6}+z_{5})/(x_{5}-x_{6})$. Hence, by $R^2={x_5}^2+{z_5}^2={x_6}^2+{z_6}^2$, one computes that
\[
\frac{x_{5}+x_{6}}{\eta}
=\frac{(x_{5}z_{6}-x_{6}z_{5})(z_{5}+z_{6})}{x_{5}-x_{6}}
=R^2+z_{5}z_{6}+x_{5}x_{6}
=2R^2-\frac{(x_{5}-x_{6})^{2}+(z_{5}-z_{6})^{2}}{2}.
\]
Consequently,
\begin{align*}
G_1
&=\frac{4(x_{5}-x_{6})\bigl[2c(x_{5}+x_{6})-\eta(c^2+4x_{5}x_{6})\bigr]}
{(c^2-4{x_5}^2)(c^2-4{x_6}^2)}\\
&=\frac{4(x_{5}-x_{6})\eta\bigl[c(4R^2-c-(x_{5}-x_{6})^{2}-(z_{5}-z_{6})^{2})-4x_{5}x_{6}\bigr]}
{(c^2-4{x_5}^2)(c^2-4{x_6}^2)}.
\end{align*}
The denominator is positive, and \eqref{7.11} gives
$4R^2-c-(x_{5}-x_{6})^{2}-(z_{5}-z_{6})^{2}>4R^2-2c=2(R^2-1)>0$.
Therefore,
\begin{equation}\label{7.12}
G_1>0.
\end{equation}

To transfer this inequality to the exponent $3/2$, let
$\ell=x_{5}-x_{6}$ and $z_\pm=c\pm(x_{5}+x_{6})$.
Then $0<\ell<z_-<z_+$, because $z_{-}-\ell=c-2x_{5}=1+R^{2}-2x_{5}=(x_5-1)^{2}+{z_5}^{2}>0$.
For $\varsigma>\ell$, set
$L_\beta(\varsigma)=(\varsigma-\ell)^{-\beta}-(\varsigma+\ell)^{-\beta}$.
Regrouping the four terms in $G_\beta$ yields
\begin{equation}\label{7.13}
G_\beta=(1-\eta)L_\beta(z_-)
-(1+\eta)L_\beta(z_+).
\end{equation}
Define $E(\varsigma)=L_{3/2}(\varsigma)/L_1(\varsigma)$.
Applying the fundamental theorem of calculus to the divided
difference of $y\mapsto y^{3/2}$, we find
\[
E(\varsigma)=\frac{\Big((\varsigma-\ell)^{-1}\Big)^{3/2}-\Big((\varsigma+\ell)^{-1})\Big)^{3/2}}{(\varsigma-\ell)^{-1}-(\varsigma+\ell)^{-1}}
=\frac32\int_0^1
\Big(t(\varsigma-\ell)^{-1}+(1-t)(\varsigma+\ell)^{-1}\Big)^{1/2}\,dt.
\]
Thus $E(\varsigma)>0$ and $E$ is strictly decreasing on $(\ell,\infty)$.
In particular, $E(z_-)>E(z_+)$.
Using \eqref{7.13}, $0<\eta<1$, and
\eqref{7.12}, we conclude that
\[
\begin{aligned}
G_{3/2}
&=(1-\eta)L_1(z_-)E(z_-)
 -(1+\eta)L_1(z_+)E(z_+)\\
&>E(z_+)\bigl[(1-\eta)L_1(z_-)
 -(1+\eta)L_1(z_+)\bigr]=E(z_+)G_1>0.
\end{aligned}
\]
This contradicts $G_{3/2}=0$.
Hence $x_{5}=-x_{6}$, and \eqref{7.6} then gives $z_{5}=z_{6}$.
Finally, the identity $m_{5}z_{5}J(x_{5})=m_{6}z_{6}J(-x_{6})$ in \eqref{7.7},
together with $z_{5}J(x_{5})>0$, yields $m_{5}=m_{6}$.

\medskip

{\bf To prove Theorem \ref{thm2.5} (3):}
When $q_5$ belongs to the $YOZ$-plane, then using Theorem \ref{thm2.5} (1), we have $x_6 = 0$, so that $q_6 = (0, y_6, z_6)$. The remainder of the proof of Theorem \ref{thm2.5} (3) is divided into three cases.

\textbf{Case I.} Assume that $y_5=0$. Recall that $q_5$ is not in the plane containing the rhombus, and that $x_5=0$, $z_5>0$, and $z_6>0$. Then using {\bf Step 1} from the proof of Proposition \ref{prop6.3}, it follows that the rhombus must be a square. This contradicts the assumption in Theorem \ref{thm2.5}. Therefore, $y_5\neq0$.

\textbf{Case II.} Assume $y_5>0$, then inequality (\ref{6.6}) holds. Thus, by the first step in the proof of Proposition \ref{prop6.2}, we have $y_6<0$.

\textbf{Case III.} Assume $y_5<0$, then
\begin{eqnarray*}
\displaystyle\frac{1}{\left[\sqrt{{x_5}^2+(y_5-a)^2+{z_5}^{2}}\right]^3}-\frac{1}{\left[\sqrt{{x_5}^2+(y_5+a)^2+{z_5}^{2}}\right]^3}<0.
 \end{eqnarray*}
Hence, combining $z_{6}/z_{5}>0$, $a>0$ and $a\neq1$, $m_{5}>0$, and $m_{6}>0$, for equation (\ref{6.4}) to hold, we must have
 \begin{eqnarray*}
\displaystyle\frac{1}{\left[\sqrt{{x_6}^2+(y_6-a)^2+{z_6}^{2}}\right]^3}-\frac{1}{\left[\sqrt{{x_6}^2+(y_6+a)^2+{z_6}^{2}}\right]^3}>0,
 \end{eqnarray*}
which implies that $y_6>0$.

From the \textbf{Cases I-III}, we have $y_5y_6<0$.

Without loss of generality, we write
\[
y_5>0,\qquad y_6<0,
\qquad z_5>0,\qquad z_6>0.
\]
Thus $q_5=(0,y_{5},z_{5})$ and $q_6=(0,y_{6},z_{6})$.

Apply the orthogonal transformation
\[
T(x,y,z)=(y,x,z).
\]
Since $T$ preserves Euclidean distances, the transformed positions
$Tq_k$ again form a central configuration with the same masses and
multiplier. Explicitly,
\[
Tq_1=(-a,0,0), \,  Tq_2=(0,1,0), \,Tq_3=(a,0,0), \, Tq_4=(0,-1,0), \, Tq_5=(y_{5},0,z_{5}), \, Tq_6=(y_{6},0,z_{6}).
\]
Now scale all positions by $1/a$ and set $\widetilde q_k=a^{-1}Tq_k$.
Central configurations are invariant under positive homotheties: if
$q_k$ satisfies \eqref{1.2} with multiplier $\lambda$, then $\mu q_k$
satisfies the corresponding equations with multiplier $\lambda/\mu^3$.
Hence the scaled configuration is again a central configuration, with
multiplier $a^3\lambda$.

Relabel the first four bodies by
\[
Q_1=\widetilde q_4,\qquad
Q_2=\widetilde q_3,\qquad
Q_3=\widetilde q_2,\qquad
Q_4=\widetilde q_1,
\]
and put $Q_5=\widetilde q_5$, $Q_6=\widetilde q_6$. Then
\begin{eqnarray*}
\left\{
  \begin{array}{ll}
Q_1=\displaystyle\left(0,-\frac1a,0\right), \,\,
Q_2=(1,0,0), \,\, Q_3=\displaystyle\left(0,\frac1a,0\right),\\ [0.3cm]
Q_4=(-1,0,0), \,\, Q_5=\displaystyle\left(\frac{y_{5}}{a},0,\displaystyle\frac{z_{5}}{a}\right), \,\, Q_6=\displaystyle\left(\frac{y_{6}}{a},0,\frac{z_{6}}{a}\right),
  \end{array}
\right.
\end{eqnarray*}
where the corresponding six masses are $M_1,\,M_2,\cdots,M_6$.
Based on Theorem \ref{thm2.5} (2), then the relabeled masses satisfy
\[
M_1=M_3=m_2=m_4, \qquad M_2=M_4=m_1=m_3, \qquad m_5=m_6.
\]
$m_2=m_4$ and $m_1=m_3$.
Thus Theorem \ref{thm2.5} (2) applies with parameter $1/a>0$ and $\sqrt{}3/3<1/a<1$.
It follows that
\[
\frac{y_{5}}{a}=\frac{-y_{6}}{a},
\qquad
\frac{z_{5}}{a}=\frac{z_{6}}{a}, \qquad 1<a<\sqrt{3}.
\]
Since $a>0$, we obtain $y_{5}=-y_{6}$ and $z_{5}=z_{6}$.
 $\Box$

\section{Proof of Theorem \ref{thm2.6}}\label{sec8}\setcounter{equation}{0}

From the proof in Theorem \ref{thm2.5} (3), for every $a\in(1,\sqrt{3})$, if we prove the existence for the spatial central configuration with 
\[
x_{5}=x_{6}=0, y_{5}=-y_{6}>0, z_{5}=z_{6}>0, m_{1}=m_{3}, m_{2}=m_{4}, m_{5}=m_{6},
\]
then using the orthogonal transformation
\[
T(x,y,z)=(y,x,z)
\]
and scaling all positions of the six masses by $1/a$, we can directly obtain the existence for the spatial central configuration for every $a\in(\sqrt{3}/3,1)$. 
So, in the following, we only need to deal with the case of $a\in(1,\sqrt{3})$.

Assume that $x_{5}=x_{6}=0$, $y_{5}=-y_{6}>0$, $z_{5}=z_{6}>0$,
$m_1=m_3$, $m_2=m_4$, $m_5=m_6$ and $1<a<\sqrt{3}$. We consider the Newtonian
central configuration equations
\begin{equation}\label{8.1}
\sum\limits_{{j\neq k\atop{{1\leq j\leq 6}}}}
\frac{m_jm_k(q_j-q_k)}{|q_j-q_k|^3}
+\lambda m_k(q_k-q_0)=0,
\qquad k=1,\ldots,6,
\end{equation}
where
$q_0=M^{-1}\sum_{j=1}^6m_jq_j$ and $M=\sum_{j=1}^6m_j$.

Since multiplication of all masses by the same positive constant does not
change the geometry, we normalize
\[
m_2=m_4=1
\]
and write
\[
m_1=m_3=m,\qquad m_5=m_6=\phi.
\]
Set $S=m+1+\phi$ and $r^2={y_5}^2+{z_5}^2$. Then
$q_0=\Big(0,0,(\phi z_{5})/S\Big)$.

Introduce
\[
L=(1+a^2)^{-3/2},\qquad
N=(1+r^2)^{-3/2},\qquad \alpha_{-}=(a^2+r^2-2ay_5)^{-3/2},\qquad
\alpha_+=(a^2+r^2+2ay_5)^{-3/2}.
\]
Since $a,y_{5}>0$, one has $\alpha_->\alpha_+$.

The $z$-component of the equation at $q_2$ gives
\begin{equation}\label{8.2}
\lambda=2SN.
\end{equation}
The $z$-component of the equation at $q_1$ therefore yields the purely
geometric relation
\begin{equation}\label{8.3}
\alpha_-+\alpha_+=2N.
\end{equation}
Moreover, using $\lambda/M=\lambda/2S=N$ and $D^{-3}_{52}=({y_5}^2+{z_5}^2+1)^{-3/2}=N$, the $x$-component at $q_2$ gives
\begin{equation}\label{8.4}
m(L-N)=N-\frac18.
\end{equation}
Finally, set
\[
\Delta=\alpha_{-}-\alpha_{+}>0,
\qquad
H=\frac{1}{4{y_5}^3}-2N.
\]
Then the $y$-component at $q_5$ becomes
\begin{equation}\label{8.5}
\phi H=\frac{ma\Delta}{{y_5}},
\end{equation}
and the $y$-component at $q_1$, after using
\eqref{8.2}--\eqref{8.4}, becomes
\begin{equation}\label{8.6}
2L-\frac14
+m\left(\frac{1}{4a^3}-2L-\frac{\Delta^2}{H}\right)=0.
\end{equation}

Fix $a\in(1,\sqrt3)$. For $r\in(a,\sqrt3)$, define
\[
P=a^2+r^2,\qquad W=1+r^2
\]
and
\[
\Phi_{a,r}(y_5)
=(P-2ay_{5})^{-3/2}+(P+2ay_{5})^{-3/2},
\qquad 0\le y_5\le r.
\]
Then \eqref{8.3} is equivalent to
\begin{equation}\label{8.7}
\Phi_{a,r}(y_5)=2W^{-3/2}.
\end{equation}

In what follows, we divide the remaining proof into three parts.

{\bf Part I:}
For every $a\in(1,\sqrt3)$ and every $r\in(a,\sqrt3)$, we show that equation
\eqref{8.7} admits a unique solution $y_{5}=s_a(r)\in(0,r)$.
Moreover, $s_a(r)$ depends continuously, indeed smoothly, on $r$.

In fact, at $y_{5}=0$,
\[
\Phi_{a,r}(0)=2(a^2+r^2)^{-3/2}
<2(1+r^2)^{-\frac32}=2W^{-3/2},
\]
because $a>1$. At $y_{5}=r$,
\[
\Phi_{a,r}(r)=(r-a)^{-3}+(r+a)^{-3}.
\]
Since $0<r-a<\sqrt3-1<1$, the first term is larger than $1$, whereas
$2W^{-3/2}<1$. Hence
$\Phi_{a,r}(r)>2W^{-3/2}$.

Furthermore,
\[
\Phi'_{a,r}(y_5)
=3a(P-2ay_5)^{-5/2}-3a(P+2ay_5)^{-5/2}>0
\]
for $y_{5}\in(0,r)$. Thus the intermediate value theorem and strict
monotonicity give a unique solution $s_a(r)\in(0,r)$. Since
$\partial_{y_5}\Phi_{a,r}(s_a(r))>0$, the implicit function theorem gives
smooth dependence on $r$.

{\bf Part II:}
For every $r\in(a,\sqrt3)$, we prove that the solution in {\bf Part I} satisfies
$4{s_a(r)}^2<1+r^2$.
Consequently,
\[
H=\frac{1}{4{s_a(r)}^3}-2(1+r^2)^{-3/2}>0.
\]

In fact, let
\[
s_0=\frac{\sqrt{1+r^2}}{2}=\frac{\sqrt W}{2}.
\]
Since $r>1$, one has $0<s_0<r$, which implies $a^{2}+r^{2}-2as_{0}>0$. Put $u=a/\sqrt W$. Because
$1<a<r<\sqrt3$, we have
$1/2<u<\sqrt{3}/2$.
Moreover, since $W<4$,
\[
\frac{a^2+r^2}{W}=u^2+1-\frac1W<u^2+\frac34.
\]
Therefore
\[
W^{3/2}\Phi_{a,r}(s_0)
>
\left(u^2-u+\frac34\right)^{-3/2}
+
\left(u^2+u+\frac34\right)^{-3/2}.
\]
For $u\in(1/2,\sqrt3/2)$,
\[
u^2-u+\frac34\le \frac{3-\sqrt3}{2},
\qquad
u^2+u+\frac34\le \frac{3+\sqrt3}{2}.
\]
Hence
\[
W^{3/2}\Phi_{a,r}(s_0)
>
\left(1+\frac1{\sqrt3}\right)^{3/2}
+
\left(1-\frac1{\sqrt3}\right)^{3/2}\geq \left(2\sqrt{\frac{2}{3}}\right)^{3/2}=\left(\frac{16\sqrt{2}}{3\sqrt{3}}\right)^{1/2}>2.
\]
 Thus
$\Phi_{a,r}(s_0)>2W^{-3/2}$. Since $\Phi_{a,r}$ is strictly increasing
and $y_5=s_a(r)$ is the unique solution of \eqref{8.7}, it follows that
$s_a(r)<s_0$. Hence $4{s_a(r)}^2<1+r^2$, which is equivalent to $H>0$.

{\bf Part III:} For every $1<a<\sqrt3$, we prove the existence of the spatial central configuration.

For $r\in(a,\sqrt3)$, define
\begin{equation}\label{8.8}
m(r)=\frac{N-\frac18}{L-N},
\end{equation}
where $L=(1+a^2)^{-3/2}$ and
$N=(1+r^2)^{-3/2}$.
Since $a<r<\sqrt3$, we have
$1/8<N<L$,
and therefore $m(r)>0$. Recall the notation $m_{5}=m_{6}=\phi$. Choose
\begin{equation}\label{8.9}
\phi=\frac{m(r)a\Delta}{s_a(r)H}:=\phi(r).
\end{equation}
Then $s_a(r)>0$ and the conclusion of {\bf Part II} give $\phi(r)>0$.

Now set
\begin{equation}\label{8.10}
G_a(r)
=
2L-\frac14
+m(r)
\left(
\frac{1}{4a^3}-2L-\frac{\Delta^2}{H}
\right).
\end{equation}
Then \eqref{8.6} is exactly $G_a(r)=0$.

Fix $a\in(1,\sqrt3)$. By {\bf Parts I and II}, the
function $G_a$ in \eqref{8.10} is continuous on $(a,\sqrt3)$.

As $r\downarrow a$, one has $N\to L$. Since $L>1/8$,
\[
m(r)=\frac{N-\frac18}{L-N}\longrightarrow+\infty.
\]
Furthermore,
\[
\frac{1}{4a^3}-2L<0.
\]
Indeed, this inequality is equivalent to
$(1+a^2)^{3/2}<8a^3$, which follows from $a>1$. Recall that $y_5=s_a(r)$ and $4{s_a(r)}^2<1+r^2$. Then
$\Delta^2/H>0$, which implies the quantity inside the parentheses in \eqref{8.10}
is bounded above by the fixed negative number
$1/(4a^3)-2L$. Hence
\begin{equation}\label{8.11}
G_a(r)\longrightarrow-\infty
\qquad\text{as }r\downarrow a.
\end{equation}

Next let $r\uparrow\sqrt3$. Then $N\to1/8$, and therefore
$m(r)\to0$. The unique solution $s_a(r)$ extends continuously to the
unique solution of \eqref{8.7} at $r=\sqrt3$. The proof of {\bf Part II}, applied at $r=\sqrt3$, shows that the limiting value
$s_*$ satisfies $s_*<1$. Consequently,
$H\longrightarrow 1/(4{s_*}^3)-1/4>0$,
so $\Delta^2/H$ remains bounded. Thus
\begin{equation}\label{8.12}
\lim_{r\uparrow\sqrt3}G_a(r)
=2L-\frac14>0,
\end{equation}
because $a<\sqrt3$.

By \eqref{8.11}, \eqref{8.12}, and the intermediate value
theorem, there exists $r_a\in(a,\sqrt3)$ such that $G_a(r_a)=0$.
Set
\[
y_5=s_a(r_a),\qquad
z_5=\sqrt{{r_a}^2-{y_5}^2},
\]
and normalize $m_2=m_4=1$. Define
\[
m_1=m_3=m(r_a)>0,
\qquad
m_5=m_6=\phi(r_a)>0,
\]
and let
\[
\lambda=2\bigl(m(r_a)+1+\phi(r_a)\bigr)N(r_a).
\]
Combining equations \eqref{8.7}--\eqref{8.9} and $G_a(r_a)=0$ hold, we deduce that for every $a\in(1,\sqrt3)$, there exist $x_{5}=x_{6}=0$, $0<y_{5}=-y_{6}=s_a(r_a)<r$, $z_{5}=z_{6}=\sqrt{{r_a}^2-{y_5}^2}>0$, and
\[
m_1=m_3=m(r_a),
\qquad
m_2=m_4=1,
\qquad
m_5=m_6=\phi(r_a)
\]
such that
independent central
configuration equations \eqref{8.2}--\eqref{8.6} hold. The remaining equations follow from the
reflection symmetries of \eqref{2.1} with $x_{5}=x_{6}=0$, $y_{5}=-y_{6}>0$ and $z_{5}=z_{6}>0$; the remaining
$z$-component at $q_5$ is also an immediate consequence of
\eqref{8.3} and \eqref{8.2}. Hence \eqref{8.1} holds. $\Box$\\

%%%%%%%%%%%%%%%%%%%

\section{Proof of Theorem \ref{thm2.7}: Existence of Perverse Solutions }\label{sec9}\setcounter{equation}{0}

First, we prove the existence of a spatial central configuration. For this particular configuration, we assume that equation (\ref{2.2}) holds, that $q_5$ is at the geometric center, and the geometric center coincides with the center of mass. In this case,
 the central configuration equations (\ref{1.2}) are reduced to four equations.
\begin{eqnarray}\label{9.1}
\left\{
  \begin{array}{llllll}
f_{x_2}=\displaystyle-\frac{ \sqrt{2}m_0}{2}-\frac{m_0}{4}-\frac{m_5}{\left({z_5}^{2}+1\right)^{\frac{3}{2}}}-\frac{m_6}{\left(m_6^{2}+1\right)^{\frac{3}{2}}}+\lambda=0,\\[0.3cm]
f_{z_1}=\displaystyle\frac{m_5 z_5}{\left({z_5}^{2}+1\right)^{\frac{3}{2}}}+\frac{m_6 z_6}{\left({z_6}^{2}+1\right)^{\frac{3}{2}}}-\frac{\lambda \left(m_5 z_5 +m_6 z_6 \right)}{4 m_0 +m_5 +m_6}=0, \\[0.3cm]
f_{z_5}=\displaystyle -\frac{4 m_0 z_5}{\left({z_5}^{2}+1\right)^{\frac{3}{2}}}+\frac{m_6 \left(z_6 -z_5 \right)}{\left(\left(z_6 -z_5 \right)^{2}\right)^{\frac{3}{2}}}+\lambda \left(z_5 -\frac{m_5 z_5 +m_6 z_6}{4 m_0 +m_5 +m_6}\right)=0,\\[0.3cm]
f_{z_6}= \displaystyle-\frac{4 m_0 z_6}{\left({z_6}^{2}+1\right)^{\frac{3}{2}}}+\frac{m_5 \left(-z_6 +z_5 \right)}{\left(\left(z_6 -z_5 \right)^{2}\right)^{\frac{3}{2}}}+\lambda \left(z_6 -\frac{m_5 z_5 +m_6 z_6}{4 m_0 +m_5 +m_6}\right)=0,
  \end{array}
\right.
\end{eqnarray}

Let $z_5=b>0$. To make $q_5$ at the geometric center, i.e., $D_{15}=D_{25}=D_{35}=D_{45}=D_{56}=\sqrt{b^2+1}$, we choose  $z_6=\sqrt{b^2+1}+b$.
In order to make the center of mass at $q_5$, from $\frac{m_1q_1+ m_2q_2 +\cdots +m_6q_6}{m_1+m_2+\cdots +m_6}= q_5$, we have $\frac{m_5z_5 +m_6z_6}{4m_0+m_5+m_6}=z_5=b$, which implies that $m_6=\frac{4m_0b}{\sqrt{b^2+1}}.$ Solve $f_{x_2}=0$ for $\lambda$, one obtains
\begin{eqnarray}\label{9.2}
  \lambda= \frac{\sqrt{2}m_0}{2}+\frac{m_0}{4}+\frac{\mathit{m_5}}{\left(b^{2}+1\right)^{\frac{3}{2}}}+\frac{4 bm_0}{\sqrt{b^{2}+1}\, \left(\left(\sqrt{b^{2}+1}+b \right)^{2}+1\right)^{\frac{3}{2}}},
\end{eqnarray}
  which is the expression in equation (\ref{2.4}).  Substituting these relations in above equations, then $f_{z_5}\equiv 0$ and
$$f_{z_1}=-\frac{2 \mathit{m_0} \left(-4+\left(\sqrt{b^{2}+1}\, b +b^{2}+1\right) \left(\sqrt{2}+\frac{1}{2}\right) \sqrt{2 b^{2}+2+2 \sqrt{b^{2}+1}\, b}\right) b}{\sqrt{2 b^{2}+2+2 \sqrt{b^{2}+1}\, b}\, \left(4 b^{2}+4 \sqrt{b^{2}+1}\, b +4\right)}=0
$$
$$f_{z_6}=\frac{2 \mathit{m_0} \left(-4+\left(\sqrt{b^{2}+1}\, b +b^{2}+1\right) \left(\sqrt{2}+\frac{1}{2}\right) \sqrt{2 b^{2}+2+2 \sqrt{b^{2}+1}\, b}\right) \sqrt{b^{2}+1}}{\sqrt{2 b^{2}+2+2 \sqrt{b^{2}+1}\, b}\, \left(4 b^{2}+4 \sqrt{b^{2}+1}\, b +4\right)}=0.
$$
These two equations imply that $b$ must satisfy equation (\ref{2.3}), which is below
$$ \widehat{g}(b)= -4+\left(\sqrt{b^{2}+1}\, b +b^{2}+1\right) \left(\sqrt{2}+\frac{1}{2}\right) \sqrt{2 b^{2}+2+2 \sqrt{b^{2}+1}\, b}=0.$$
It is easy to verify that $\widehat{g}(0)=-2+\frac{\sqrt{2}}{2}<0$, and
$$\widehat{g}(1)=-4+\left(2+\sqrt{2}\right) \left(\sqrt{2}+\frac{1}{2}\right) \sqrt{4+2 \sqrt{2}}>0. $$  The function $\widehat{g}$ is strictly increasing because the derivative $\widehat{g}\,^{\prime}(b)>0$. By the Intermediate Value Theorem and the monotonicity of the function $\widehat{g}$,  there exists a unique positive value $b_0$, numerically $b_0\approx 0.2354261907$, such that $\widehat{g}(b_0)=0$.

Therefore, we conclude that there exists a $\lambda\in\mathbb{R}$ such that equations (\ref{9.1}) hold; that is, we have proved the existence of the special central configuration.

Second, we use this special central configuration to establish the existence of perverse solutions in $\mathbb{R}^{3}$.

For this unique positive value $b_0$, we obtain $z_5=b_0$ and $z_6=\sqrt{{b_0}^2+1} +b_0$.  Then the configuration $q = [q_1, q_2, \cdots, q_6] $  defined as in equation (\ref{2.2}) forms  a central configuration for the mass vector $(m_1, m_2, m_3, m_4, m_5, m_6)$ with $m_1=m_2=m_3=m_4=m_0$, $m_6=\frac{4m_{0}b_{0}}{\sqrt{{b_0}^2+1}}$, where $m_0$ and $m_5$ are arbitrary positive values.

The location $q_5$ coincides with both the mass center (which is independent of $m_0$ and $m_5$) and   the geometric center, meaning that $D_{15}=D_{25}=D_{35}=D_{45}=D_{56}$.

For a fixed $\lambda=\lambda_0$, since $f_{z_5} \equiv 0$ and by equation (\ref{9.2}),  the homothetic collision/expansion solution generated from this particular central configuration yields a perverse solution for the mass sets by choosing $m_0 \in  (0, 0.8500232075\lambda_0)$ and  $m_5= \frac{\lambda_0-1.176438468m_0}{0.9222711668}>0$.

It is clear that the sum of masses for this two set of masses are different. In fact for fixed $\lambda =\lambda_0$ and the total mass $M=M_0$, there is a unique set of mass $(m_0, m_5)$ that satisfies the linear equation system : $$  \lambda_0=1.176438468m_0 + 0.9222711667m_5$$  and $$M_0=4.9166445852m_0 +m_5.$$
Therefore the homothetic total collision/expansion solution generated from this central configuration for an arbitrarily fixed $\lambda$ is a perverse solution, but it is not a really perverse solution.

This completes the proof for Theorem \ref{thm2.7}. $\Box$

\section*{CRediT authorship contribution statement} \textbf{Liang Ding:} Conceptualization, Formal analysis, Methodology, Validation, Visualization, Writing - original draft, Writing - review \& editing. \textbf{Jin Wang:} Conceptualization, Methodology, Writing - original draft. \textbf{Jinrong Wang:} Conceptualization, Formal analysis, Methodology, Writing - original draft, Writing - review \& editing.
\textbf{Zhifu Xie:} Conceptualization, Formal analysis, Methodology, Supervision, Validation, Writing - original draft, Writing - review \& editing.

\section*{Declaration of competing interest} The authors declare that they have no known competing financial
interests or personal relationships that could have appeared to influence
the work reported in this paper.

\section*{Data availability} Data will be made available on request.

\section*{Acknowledgements} Liang Ding was partially supported by the National Natural Science Foundation of China (Grant Number: 12361022), Guizhou Provincial Basic Research Program (Natural Science, Grant Number: MS[2026]315), and Guizhou Provincial ``Jiaxiu Zhiguang'' Visiting Scholar Program. Jinrong Wang was partially supported by Guizhou Provincial Major Project of Basic Research Program,
Qiankehe Foundation (Grant Number: VZD[2026]001) and Guizhou Provincial Key Laboratory of Applied Mathematics and Computing Power \& Algorithms (Qiankehe Platform ZSYS[2025]039).
Zhifu Xie was partially supported by USA National Science Foundation (Award Number: 2447675) and USM Foundation (Wright W. and Annie Rea Cross Endowment).

\end{document}